\documentclass[onefignum,onetabnum]{siamart171218}

\usepackage{amsmath}
\usepackage{amssymb}                                    % mathematical symbols and AMS fonts
\usepackage{bbm}                                            % blackboard bold for numbers
\ifpdf
  \DeclareGraphicsExtensions{.pdf,.eps,.png,.jpg}
\else
  \DeclareGraphicsExtensions{.eps}
\fi

\newsiamremark{remark}{Remark}
\newsiamremark{hypothesis}{Hypothesis}
\newsiamremark{assumption}{Assumption}
\usepackage{bm}                                         % bold fonts

\usepackage[%pdftex%
            ]{graphicx}                                 % the graphics package
\newcommand{\C}{\mathbb{C}}

\newcommand{\E}{\mathbb{E}}
\newcommand{\N}{\mathbb{N}}

\newcommand{\R}{\mathbb{R}}
\newcommand{\Z}{\mathbb{Z}}

\def\det{\mathop\mathrm{det}\nolimits}
\def\diag{\mathop\mathrm{diag}\nolimits}
\def\dim{\mathop\mathrm{dim}\nolimits}

\def\exp{\mathop\mathrm{exp}\nolimits}

\def\Ham{\mathop\mathrm{Ham}\nolimits}
\def\id{\mathcal{I}}
\def\ker{\mathop\mathrm{ker}\nolimits}

\def\Ran{\mathop\mathrm{ran}\nolimits}

\def\Span{\mathop\mathrm{span}\nolimits}

\def\coloneqq{\mathrel{\mathop:}=}

\newcommand{\rma}{\mathrm{a}}

\newcommand{\rmc}{\mathrm{c}}
\newcommand{\rmd}{\mathrm{d}}
\newcommand{\rme}{\mathrm{e}}

\newcommand{\rmi}{\mathrm{i}}

\newcommand{\rmn}{\mathrm{n}}

\newcommand{\rmp}{\mathrm{p}}
\newcommand{\rmr}{\mathrm{r}}
\newcommand{\rms}{\mathrm{s}}

\newcommand{\calA}{\mathcal{A}}

\newcommand{\calI}{\mathcal{I}}

\newcommand{\calL}{\mathcal{L}}

\newcommand{\calO}{\mathcal{O}}
\newcommand{\calP}{\mathcal{P}}

\newcommand{\calR}{\mathcal{R}}
\newcommand{\calS}{\mathcal{S}}
\newcommand{\calT}{\mathcal{T}}

\newcommand{\vI}{\bm{\mathit{I}}}

\newcommand{\vK}{\bm{\mathit{K}}}

\newcommand{\vS}{\bm{\mathit{S}}}

\newcommand{\vx}{\bm{\mathit{x}}}

\newcommand{\vv}{\bm{\mathit{v}}}

\newcommand{\vn}{\bm{\mathit{0}}}

\headers{A reformulated Krein matrix}{T. Kapitula, R. Parker, and B. Sandstede}

\title{A reformulated Krein matrix for star-even polynomial operators with applications\thanks{Submitted to the editors 01/30/2019.
\funding{TK was partially funded by the NSF under DMS-1108783. RP was partially funded by the NSF under DMS-1148284. BS was partially funded by the NSF under DMS-1714429.}}}

\author{Todd Kapitula\thanks{Department of Mathematics and Statistics, Calvin University, Grand Rapids, MI 49546
  (\email{tmk5@calvin.edu}).}
\and %
Ross Parker\thanks{Division of Applied Mathematics, Brown University, Providence, RI 02912
  (\email{ross\_parker@brown.edu}).}
\and %
Bj\"orn Sandstede\thanks{Division of Applied Mathematics, Brown University, Providence, RI 02912
  (\email{bjorn\_sandstede@brown.edu}).}
  }

\ifpdf
\hypersetup{
  pdftitle={A reformulated Krein matrix for star-even polynomial  operators with applications},
  pdfauthor={T. Kapitula, R. Parker, and B. Sandstede}
}
\fi

\begin{document}

\maketitle

% REQUIRED
\begin{abstract}
In its original formulation the Krein matrix was used to locate the spectrum of first-order star-even polynomial operators where both operator coefficients are nonsingular. Such operators naturally arise when considering first-order-in-time Hamiltonian PDEs. Herein the matrix is reformulated to allow for operator coefficients with nontrivial kernel. Moreover, it is extended to allow for the study of the spectral problem associated with quadratic star-even operators, which arise when considering the spectral problem associated with second-order-in-time Hamiltonian PDEs. In conjunction with the Hamiltonian-Krein index (HKI) the Krein matrix is used to study two problems: conditions leading to Hamiltonian-Hopf bifurcations for small spatially periodic waves, and the location and Krein signature of small eigenvalues associated with, e.g., $n$-pulse problems. For the first case we consider in detail a first-order-in-time fifth-order KdV-like equation. In the latter case we use a combination of Lin's method, the HKI, and the Krein matrix to study the spectrum associated with $n$-pulses for a second-order-in-time Hamiltonian system which is used to model the dynamics of a suspension bridge.
\end{abstract}

% REQUIRED

% REQUIRED

\title{A reformulated Krein matrix for star-even polynomial operators with applications}

\section{Introduction}\label{sec:intro}

Herein we are generally concerned with the spectral stability of waves that arise as steady-states for a nonlinear Hamiltonian system which is either first-order or second-order in time.  There are two tools which we will use to study the spectrum. The first is the Hamiltonian-Krein index (HKI), which relates the number of negative directions associated with the linearized energy evaluated at the underlying wave to the number of (potentially) unstable point spectra (eigenvalues with positive real part).  If the HKI is zero, then under some fairly generic assumptions the underlying wave will be orbitally stable. If the HKI is positive, then it provides an upper bound on the number of unstable point eigenvalues. If it can be shown, either analytically or numerically, that there are no eigenvalues with positive real part, then the HKI provides the number of purely imaginary eigenvalues with negative Krein signature. 

The Krein signature of a simple purely imaginary eigenvalue of the linearization about a wave is defined to be positive (negative) if the Hessian of the energy, also evaluated at the wave and restricted to the corresponding eigenspace of the linearization, is positive (negative) definite. Dynamically, at the linear level, eigenvalues with negative Krein signature provide temporally oscillatory behavior in an unstable energy direction. Moreover, these are the foundational eigenvalues associated with the Hamiltonian-Hopf bifurcation.
In particular, the bifurcation can occur only if purely imaginary eigenvalues of opposite signature collide when doing some type of parameter continuation. If it can be shown all the purely imaginary eigenvalues have positive signature, then a Hamiltonian-Hopf bifurcation is not possible. A formal definition of the signature in the setting of star-even polynomial operators is provided in equation \cref{e:defkrein}.

Now, the purely imaginary eigenvalues with negative Krein signature cannot be easily detected via a visual examination of the spectra. Consequently, another tool is needed.
Here we use the Krein matrix, an eigenvalue detecting tool which can also be used to determine the Krein signature of purely imaginary point eigenvalues. The Krein matrix has properties similar to those of the Evans matrix - in particular, the determinant being zero means that an eigenvalue has been found - except that it is meromorphic instead of being analytic.  By marrying the HKI with a spectral analysis via the Krein matrix one can locate all the point spectra associated with dynamical instabilities. We will illustrate the fruit of this marriage herein by considering two problems: the spectral stability associated with small spatially periodic waves, and the location and Krein signature of small eigenvalues associated with tail-tail interactions in $n$-pulses.

We now flesh out this preliminary discussion. The linearization of the Hamiltonian system will yield a star-even operator polynomial,
\[
\calP_n(\lambda)\coloneqq\sum_{j=0}^n\lambda^j\calA_j.
\]
On some Hilbert space, $X$, endowed with inner-product,
$\langle\cdot,\cdot\rangle$, which in turn induces a norm, $\|\cdot\|$,
we assume the operator coefficients $\calA_{2\ell}$ are
Hermitian, $\calA_{2\ell}^\rma=\calA_{2\ell}$, and the operator coefficients
$\calA_{2\ell+1}$ are skew-Hermitian,
$\calA_{2\ell+1}^\rma=-\calA_{2\ell+1}$. Here we let $\calT^\rma$ denote the adjoint of the operator $\calT$. If $n=1$,
\[
\calP_1(\lambda)\psi=0\,\,\leadsto\,\,
\left(\calA_0+\lambda\calA_1\right)\psi=0.
\]
Assuming $\calA_1$ is invertible, this spectral problem is equivalent to,
\[
\calA_1^{-1}\calA_0\psi=\gamma\psi,\quad\gamma=-\lambda,
\]
which, since $\calA_1^{-1}$ is skew-Hermitian and $\calA_0$ is Hermitian, is the canonical form for a Hamiltonian eigenvalue problem. Indeed, while we will not go into the details here, it is possible via a change of variables to put any star-even problem into canonical form, see \cite[Section~3]{kapitula:iif13} and the references therein. For our purposes it is best to leave the problem in its original formulation.

Values $\lambda_0$ for which the polynomial
$\calP_n(\lambda_0)$ is singular will be called \textit{polynomial
eigenvalues}. Because of these assumed coefficient properties, the polynomial
eigenvalues are symmetric with respect to the imaginary axis of the complex
plane. The eigenvalue symmetry follows from,
\[
\calP_n(\lambda)^\rma=\calP_n(-\overline{\lambda}),
\]
so $\lambda$ being a polynomial eigenvalue implies $-\overline{\lambda}$ is also a polynomial eigenvalue.
In order to ensure there are no polynomial eigenvalues at infinity, we assume
$\calA_n$ is invertible.

More can be said about the set of polynomial eigenvalues under compactness
assumptions (which will henceforth be assumed, except for the example considered in \cref{s:6}). Suppose the Hermitian
operator $\calA_0$ has compact resolvent, so the eigenvalues for this
operator coefficient are real, semi-simple, and have finite multiplicity. Let
$P_{\calA_0}:X\mapsto\ker(\calA_0)$ be the orthogonal projection, and set
$P_{\calA_0}^\perp=\calI-P_{\calA_0}:X\mapsto\ker(\calA_0)^\perp$. Assuming
the operators,
\[
\left(P_{\calA_0}^\perp\calA_0P_{\calA_0}^\perp\right)^{-1}
P_{\calA_0}^\perp\calA_jP_{\calA_0}^\perp:\ker(\calA_0)^\perp\mapsto\ker(\calA_0)^\perp,\quad
j=1,\dots,n,
\]
are compact, the spectrum for $\calP_n(\lambda)$ is point spectra only
\cite[Remark~2.2]{bronski:aii14}. Moreover, each polynomial eigenvalue has
finite multiplicity, and infinity is the only possible accumulation point for
the polynomial eigenvalues.

Regarding the number of unstable polynomial eigenvalues, i.e., those
polynomial eigenvalues with positive real part, the total number can be
bounded above via the Hamiltonian-Krein index (HKI). Let $k_\rmr$ denote the
total number (counting multiplicity) of real and positive polynomial
eigenvalues, and let $k_\rmc$ be the total number (counting multiplicity) of
polynomial eigenvalues with positive real part and nonzero imaginary part.
The total number of unstable polynomial eigenvalues is $k_\rmr+k_\rmc$.

The HKI also takes into account a subset of purely imaginary polynomial
eigenvalues; namely, those with negative Krein signature. For each purely
imaginary and nonzero eigenvalue, $\rmi\lambda_0$ with $\lambda_0\in\R$, with
associated eigenspace $\E_{\rmi\lambda_0}$, set
\begin{equation}\label{e:defkrein}
k_\rmi^-(\rmi\lambda_0)=\rmn\left(-\lambda_0\left[\rmi P_n'(\rmi\lambda_0)\right]|_{\E_{\rmi\lambda_0}}\right).
\end{equation}
Here $\rmn(\vS)$ denotes the number of negative eigenvalues for the
Hermitian matrix $\vS$, and $-\lambda_0\rmi P_n'(\rmi\lambda_0)|_{\E_{\rmi\lambda_0}}$ is the Hermitian matrix
formed by the representation of the Hermitian operator
$-\rmi\lambda_0P_n'(\rmi\lambda_0)$ restricted to the eigenspace
$\E_{\rmi\lambda_0}$. If the polynomial eigenvalue is simple with associated
eigenvector $u_{\rmi\lambda_0}$, then
\[
k_\rmi^-(\rmi\lambda_0)=
\rmn\left(\lambda_0\langle-\rmi P_n'(\rmi\lambda_0)u_{\rmi\lambda_0},u_{\rmi\lambda_0}\rangle\right);
\]
in particular, if $n=1$ then it takes the more familiar form,
\[
k_\rmi^-(\rmi\lambda_0)=
\rmn\left(\langle\calA_0u_{\rmi\lambda_0},u_{\rmi\lambda_0}\rangle\right).
\]
See \cref{s:22} for more details.
The nonnegative integer $k_\rmi^-(\rmi\lambda_0)$ is the negative Krein index
associated with the purely imaginary eigenvalue. If
$k_\rmi^-(\rmi\lambda_0)=0$, the polynomial eigenvalue is said to have
positive Krein signature; otherwise, it has negative Krein signature. The
total negative Krein index is the sum of the individual Krein indices,
\[
k_\rmi^-=\sum k_\rmi^-(\rmi\lambda_0).
\]

Regarding $k_\rmi^-$, consider the
collision of two simple polynomial eigenvalues on the imaginary axis. If they both have the same
signature, then after the collision they will each remain purely imaginary.
On the other hand, if they have opposite Krein signature, then it will
generically be the case that after the collision the pair will have nonzero
real part, which due to the spectral symmetry means that one of the
polynomial eigenvalues will have positive real part. This is the so-called Hamiltonian-Hopf bifurcation. In the case of $n=1$ the interested reader
should consult \cite[Chapter~7.1]{kapitula:sad13} for more details regarding
the case of the collision of two simple polynomial  eigenvalues, and
\cite{kapitula:tks10,vougalter:eoz06} for the case of higher-order
collisions.
The case of $n\ge2$ can be reformulated as an $n=1$ problem, see \cite{kapitula:iif13} and the references therein.
Note that if $k_\rmi^-=0$, then no polynomial eigenvalues will leave the imaginary axis.

The HKI is defined to be the sum of the three indices,
\[
K_{\Ham}=k_\rmr+k_\rmc+k_\rmi^-.
\]
The HKI is intimately related to the operator coefficients.
For the sake of exposition, first suppose $\calA_0,\calA_n$ are
nonsingular.
If $X=\C^N$, i.e., the operator is actually a star-even matrix polynomial with $nN$
polynomial eigenvalues,
\[
K_{\Ham}=\begin{cases}
\rmn(\calA_0)+(\ell-1)N,\quad&n=2\ell-1\\
\rmn(\calA_0)+\rmn\left((-1)^{\ell-1}\calA_{n}\right)+(\ell-1)N,\quad&n=2\ell
\end{cases}
\]
\cite[Theorem~3.4]{kapitula:iif13}. If $n\ge3$ the upper bound for the total
number of unstable polynomial eigenvalues depends upon the dimension of the
space; consequently, taking the limit $N\to+\infty$ provides no meaningful
information regarding the limiting case of operator coefficients which are
compact operators. Consequently, we henceforth assume $n\in\{1,2\}$.

Now, suppose $\calA_0$ has a nontrivial kernel, but that the highest-order coefficient is nonsingular.
If $n=1$, then under the widely
applicable assumptions,
\begin{enumerate}
\item $\displaystyle{\calA_1:\ker(\calA_0)\mapsto\ker(\calA_0)^\perp}$
\item $\displaystyle{\calA_1\calA_0^{-1}\calA_1|_{\ker(\calA_0)}}$ is invertible,
\end{enumerate}
we know,
\begin{equation}\label{e:11}
K_{\Ham}=\rmn(\calA_0)-\rmn\left(-\calA_1\calA_0^{-1}\calA_1|_{\ker(\calA_0)}\right),
\end{equation}
see \cite{haragus:ots08,pelinovsky:ilf05} and the references therein. Regarding the operator $\calA_1$, the case where
\begin{enumerate}
  \item there is a nontrivial kernel, but where the rest of spectrum is otherwise
uniformly bounded away from the origin, is covered in
\cite{deconinck:ots15,kapitula:sif12} and
\cite[Chapter~5.3]{kapitula:sad13}
  \item where the spectrum
which is not bounded away from the origin is considered in
\cite{kapitula:ahk14,pelinovsky:sso14}.
\end{enumerate}
If $n=2$, then upon replacing condition (b) above with,
\begin{enumerate}\addtocounter{enumi}{1}
%\item $\displaystyle{\calA_1:\ker(\calA_0)\mapsto\ker(\calA_0)^\perp}$
\item
    $\displaystyle{\left(\calA_2-\calA_1\calA_0^{-1}\calA_1\right)|_{\ker(\calA_0)}}$
    is invertible,
\end{enumerate}
we know,
\begin{equation}\label{e:12}
K_{\Ham}=\rmn(\calA_0)+\rmn(\calA_2)-
\rmn\left(\left[\calA_2-\calA_1\calA_0^{-1}\calA_1\right]|_{\ker(\calA_0)}\right),
\end{equation}
see \cite{bronski:aii14}.

The goal of this paper is to construct a square matrix-valued function, say
$\vK(\lambda)$, which has the properties that for
$\lambda\in\rmi\R$,
\begin{enumerate}
\item $\vK(\lambda)$ is Hermitian and meromorphic
\item $\det\vK(\lambda)=0$ only if $\lambda$ is a polynomial eigenvalue
\item $\vK(\lambda)$ can be used to determine the Krein signature of a
    polynomial eigenvalue.
\end{enumerate}
The matrix $\vK(\lambda)$ is known as the Krein matrix. The properties (a) and (b) listed
above are reminiscent of those possessed by the Evans matrix, except that
the Evans matrix is analytic \cite[Chapters~8-10]{kapitula:sad13}. Regarding (b) and (c), since the determinant of a matrix is equal to the product of its eigenvalues, property (b) is satisfied if at least one of the eigenvalues is zero. Henceforth, we will call the eigenvalues of the Krein matrix, say $r_j(\lambda)$, the Krein eigenvalues. The determination of the Krein signature of a purely imaginary polynomial
eigenvalue takes place through the Krein eigenvalues.
If $r_j(\lambda_0)=0$ for some $\lambda_0\in\rmi\R$, the Krein
signature is found by considering the sign of $r_j'(\lambda_0)$. Thus, via a plot of the Krein eigenvalues one can graphically determine the signature of a purely imaginary polynomial eigenvalue through the slope of the curve at a zero. The interested reader should consult the beautiful paper by Koll\'ar and Miller \cite{kollar:gks14} for,
\begin{enumerate}
\item a graphical perspective on the Krein signature using the eigenvalues of the self-adjoint operator, $\calA_0+z(\rmi\calA_1)$, for $z\in\R$
\item Hamiltonian instability index results which arise from this graphical perspective.
\end{enumerate}
A significant difference between our approach and that of \cite{kollar:gks14} is the number of Krein eigenvalues to be graphed; in particular, our approach gives a finite number, whereas the approach of \cite{kollar:gks14} yields a number equal to the number of eigenvalues for $\calA_0$.

The Krein matrix was first constructed for linear
polynomials of the canonical form,
\[
\calP_1(\lambda)=\left(\begin{array}{cc}\calL_+&0\\0&\calL_-\end{array}\right)
+\lambda\left(\begin{array}{rr}0&\calI\\-\calI&0\end{array}\right),
\]
where $\calL_\pm$ are invertible Hermitian operators with compact resolvent, and $\calI$ denotes the identity operator,
see \cite{kapitula:tks10,li:sso98}.
Recent applications of the Krein matrix include a new proof of the
Jones-Grillakis instability criterion,
\[
k_\rmr\ge|\rmn(\calL_-)-\rmn(\calL_+)|,
\]
as well as a study of the spectral problem for waves to a mathematical model
for Bose-Einstein condensates \cite{kapitula:tkm13,kapitula:sif12}.

The paper is organized as follows. In \cref{s:2} the Krein matrix is constructed for star-even polynomial operators of any degree. In particular, the previous invertibility assumption on $\calA_0$ is removed.
In \cref{s:3} the properties of the Krein eigenvalues are deduced; in particular, their relation to the Krein signature of purely imaginary polynomial eigenvalues is given. In \cref{s:4} the Krein eigenvalues are used to study the Hamiltonian-Hopf bifurcation problem associated with small periodic waves. While the underlying wave is small, it is possible for the polynomial eigenvalues to have $\calO(1)$ imaginary part (see \cite{deconinck:hfi16,kollar:dco19,trichtchenko:sop18}
for a similar study using a different approach). In \cref{s:5} we show how the Krein matrix can be used to locate small eigenvalues which arise from some type of bifurcation. However, the analysis does not use perturbation theory, so it is consequently possible to use the resulting Krein matrix to consider spectral stability for multi-pulse problems, where the small eigenvalues arise from the exponentially small tail-tail interactions of a translated base pulse.Finally, in \cref{s:6} we use the Krein matrix to study the spectral problem associated with $n$-pulse solutions to the suspension bridge equation, which is a second-order-in-time Hamiltonian PDE.

\vspace{1mm}
\noindent\textbf{Acknowledgements.} The authors would like to thank the referees for their careful reading of the original manuscript, and their helpful suggestions and constructive critique. We believe that this revision is a substantial improvement over the original because of their work.

\section{The Krein matrix}\label{s:2}

The Krein matrix allows us to reduce the infinite-dimensional eigenvalue problem,
\[
\calP_n(\lambda)\psi=0,
\]
to a finite-dimensional problem,
\[
\vK_S(\lambda)\vx=\vn.
\]
Here $\vK_S(\lambda)$ is the (square) Krein matrix. Whereas the original star-even operator is analytic in the spectral parameter, the Krein matrix is meromorphic with poles on the imaginary axis. The presence of these poles is the key to using the Krein matrix to determine the Krein signature of a purely imaginary eigenvalue.

\subsection{General construction}\label{s:21}

Let $S\subset X$ be a finite-dimensional subspace of dimension $n_S$ with
orthonormal basis $\{s_j\}$, and let $P_S:X\mapsto X$ be the orthogonal
projection, i.e.,
\[
P_Su=\sum_{j=1}^{n_S}\langle u,s_j\rangle s_j.
\]
Denote the complementary orthogonal projection as
\[
P_{S^\perp}\coloneqq\id-P_S,
\]
and write
\[
u=s+s^\perp,\,\,\mathrm{with}\,\, P_Su=s,\,\,P_{S^\perp}u=s^\perp.
\]

In constructing the subspace-dependent Krein matrix, $\vK_S(\lambda)$, for
the polynomial eigenvalue problem, we will extensively use the orthogonal
projections. We first rewrite the polynomial eigenvalue problem,
\begin{equation}\label{e:pp2}
\calP_n(\lambda)s+\calP_n(\lambda)s^\perp=0.
\end{equation}
Applying the complementary projection to \cref{e:pp2} yields
\begin{equation}\label{e:pp2a}
P_{S^\perp}\calP_n(\lambda)s+P_{S^\perp}\calP_n(\lambda)P_{S^\perp}s^\perp=0.
\end{equation}
The operator $P_{S^\perp}\calP_n(\lambda)P_{S^\perp}:S^\perp\mapsto S^\perp$ is a star-even polynomial operator. Consequently, it has the same spectral properties as the original star-even operator; in particular, it is invertible except for a countable number of spectral values.
If $\lambda$ is not a polynomial eigenvalue for the operator
$P_{S^\perp}\calP_n(\lambda)P_{S^\perp}$, then we can invert to write
\[%\begin{equation}\label{e:pp3}
s^\perp=-(P_{S^\perp}\calP_n(\lambda)P_{S^\perp})^{-1}P_{S^\perp}\calP_n(\lambda)s,
\]%\end{equation}
which leads to,
\begin{equation}\label{e:pp3}
s^\perp=
P_{S^\perp}s^\perp=
-P_{S^\perp}(P_{S^\perp}\calP_n(\lambda)P_{S^\perp})^{-1}P_{S^\perp}\calP_n(\lambda)s.
\end{equation}

If we take the inner-product of \cref{e:pp2} with a basis element $s_j$, we
get
\[
\langle s_j,\calP_n(\lambda)s\rangle+\langle s_j,\calP_n(\lambda)s^\perp\rangle=0.
\]
Substitution of the expression in \cref{e:pp3} into the above provides,
\[
\langle s_j,\calP_n(\lambda)s\rangle-
\langle s_j,\calP_n(\lambda)P_{S^\perp}(P_{S^\perp}\calP_n(\lambda)P_{S^\perp})^{-1}P_{S^\perp}\calP_n(\lambda)s\rangle=0.
\]
Writing
\[
s=\sum_{j=1}^{n_S}x_js_j,
\]
the above expression becomes
\begin{equation}\label{e:pp3a}
\vK_S(\lambda)\vx=\vn,
\end{equation}
where the Krein matrix $\vK_S(\lambda)\in\C^{n_S\times n_S}$ has the form
\[
\vK_S(\lambda)=\calP_n(\lambda)|_S-%
\calP_n(\lambda)P_{S^\perp}(P_{S^\perp}\calP_n(\lambda)P_{S^\perp})^{-1}P_{S^\perp}\calP_n(\lambda)|_{S},
\]
where we use the notation
\[
\left(\calT|_S\right)_{ij}=\langle s_i,\calT s_j\rangle.
\]
%
%%so the nomenclature makes more sense.
%The Krein matrix is meromorphic, and has poles at the polynomial eigenvalues
%of the operator $P_{S^\perp}\calP_n(\lambda)P_{S^\perp}$.
In conclusion, polynomial
eigenvalues for the original problem are found via solving \cref{e:pp3a},
which means
\[
\det\vK_S(\lambda)=0,\quad\mathrm{or}\quad\vx=\vn.
\]

What does it mean if $\lambda_0$ is a polynomial eigenvalue with $\vx=\vn$?
In this case the associated eigenfunction for the polynomial eigenvalue,
$u_0$, satisfies
\[
P_Su_0=0,\quad P_{S^\perp}u_0=u_0.
\]
Going back to \cref{e:pp2} and \cref{e:pp2a} we see
\[
P_n(\lambda_0)P_{S^\perp}u_0=0\quad\leadsto\quad
P_{S^\perp}P_n(\lambda_0)P_{S^\perp}u_0=0.
\]
In other words, $\lambda_0$ is also a polynomial eigenvalue for the operator
$P_{S^\perp}P_n(\lambda_0)P_{S^\perp}$. Thus, if $\lambda_0$ is a polynomial
eigenvalue for which the associated eigenfunction resides in
$S^\perp,\,\lambda_0$ is also a pole for the Krein matrix. Consequently, we
cannot expect to capture such polynomial eigenvalues by solving
$\det\vK_S(\lambda)=0$. This fact will motivate our later choice for the
subspace $S$, as we need to know that the polynomial eigenvalues being missed
by considering the zero set of the determinant of the Krein matrix are
somehow unimportant.

The choice of the subspace is determined by looking at the Krein index of a
purely imaginary polynomial eigenvalue, $\lambda=\rmi\lambda_0$. Letting
$\E_{\rmi\lambda_0}$ denote the generalized eigenspace, the negative Krein
index is
\[
k_\rmi^-(\rmi\lambda_0)\coloneqq
\rmn\left(-\lambda_0[\rmi P_n'(\rmi\lambda_0)]|_{\E_{\rmi\lambda_0}}\right)
\]
(see \cite{bronski:aii14}).
Since the goal is to have the Krein matrix capture all possible polynomial
eigenvalues with negative Krein index through its determinant, we then want
it to be the case that if $\rmi\lambda_0$ is a polynomial eigenvalue whose
associated eigenfunction is in $S^\perp$, then the negative Krein index is
zero. In other words, we want it to be the case that the Hermitian matrix,
$-\lambda_0[\rmi\calP_n(\rmi\lambda_0)]|_{\E_{\rmi\lambda_0}}$, is positive
definite whenever $\rmi\lambda_0$ is also a polynomial eigenvalue for the
operator $P_{S^\perp}\calP_n(\lambda)P_{S^\perp}$.

\begin{remark}
In practice, mapping $\vK(\lambda)\mapsto\lambda^\ell\vK(\lambda)$ for some
$\ell\in\N$ does not change the above property of the Krein matrix. However,
as we will see, an appropriate choice of $\ell$ gives better graphical
properties regarding the determination of those polynomial eigenvalues with
negative Krein signature.
\end{remark}

\begin{remark}
Note that if $\lambda=\rmi\lambda_0\in\rmi\R$, so that the operator $\calP_n(\rmi\lambda_0)$
is Hermitian, then for $\lambda\in\rmi\R$ the elements in the second matrix can be rewritten,
\[
\left((P_{S^\perp}\calP_n(\lambda)P_{S^\perp})^{-1}|_{P_{S^\perp}\calP_n(\lambda)S}\right)_{ij}=
\langle P_{S^\perp}\calP_n(\lambda)s_i,(P_{S^\perp}\calP_n(\lambda)P_{S^\perp})^{-1}P_{S^\perp}\calP_n(\lambda)s_j\rangle.
\]
\end{remark}

\subsection{Subspace selection}\label{s:22}

We now see how the operator coefficients may dictate the choice of the subspace $S$.
First consider the first-order operator,
\[
\calP_1(\lambda)=\calA_0+\lambda\calA_1,
\]
where $\calA_0$ is Hermitian, and $\calA_1$ is skew-Hermitian. Regarding the term associated with the calculation of the negative Krein index,
\[
-\lambda_0[\rmi\calP_1'(\rmi\lambda_0)]=-\lambda_0(\rmi\calA_1).
\]
If $\psi_0$ is an eigenfunction associated with the polynomial eigenvalue, so
$\calP_1(\rmi\lambda_0)\psi_0=0$, then
\[
-\lambda_0(\rmi\calA_1)\psi_0=\calA_0\psi_0,
\]
so we recover the ``standard" definition of the negative Krein index for
first-order star-even operators,
\[
k_\rmi^-(\rmi\lambda_0)=
\rmn\left(-\lambda_0[\rmi P_1'(\rmi\lambda_0)]|_{\E_{\rmi\lambda_0}}\right)=
\rmn\left(\calA_0|_{\E_{\rmi\lambda_0}}\right).
\]

We want the matrix
$\calA_0|_{\E_{\rmi\lambda_0}}$ to be positive definite if $\E_{\rmi\lambda_0}\subset S^\perp$. If we choose,
\[
S\coloneqq N(\calA_0)\oplus\ker(\calA_0),
\]
where $N(\calA_0)$ is the finite-dimensional negative subspace of $\calA_0$,
and $\ker(\calA_0)$ is the finite-dimensional kernel, then the fact that
$\calA_0$ is positive definite on $S^\perp$ implies that if $\rmi\lambda_0$
is a polynomial eigenvalue whose associated eigenfunction resides in
$S^\perp$, then the negative Krein index will be zero. Note that in this case $P_S$ and $P_{S^\perp}$ will be spectral projections.
Further note that with this choice of subspace that if a pole of the Krein matrix corresponds to purely imaginary polynomial
eigenvalue, then it will necessarily have positive Krein index. Consequently, all purely imaginary polynomial eigenvalues with negative Krein index will be captured by solving $\det\vK_S(\lambda)=0$.

%\subsection{Second-order operator polynomials}\label{s:23}

Now, consider the second-order operator
\[
\calP_2(\lambda)=\calA_0+\lambda\calA_1+\lambda^2\calA_2,
\]
where $\calA_0,\calA_2$ are Hermitian, and $\calA_1$ is skew-Hermitian. We
have
\[
-\lambda_0[\rmi \calP_2'(\rmi\lambda_0)]=-\lambda_0(\rmi\calA_1)+2\lambda_0^2\calA_2.
\]
If $\psi_0$ is an eigenfunction associated with the polynomial eigenvalue, so
$\calP_2(\rmi\lambda_0)\psi_0=0$,
\[
\left(-\lambda_0(\rmi\calA_1)+2\lambda_0^2\calA_2\right)\psi_0=\left(\calA_0+\lambda_0^2\calA_2\right)\psi_0.
\]
The negative Krein index can be alternatively defined,
\[
k_\rmi^-(\rmi\lambda_0)=
\rmn\left(-\lambda_0[\rmi\calP_2'(\rmi\lambda_0)]|_{\E_{\rmi\lambda_0}}\right)=
\rmn\left((\calA_0+\lambda_0^2\calA_2)|_{\E_{\rmi\lambda_0}}\right).
\]

In order for it to be the case that the matrix
$(\calA_0+\lambda_0^2\calA_2)|_{\E_{\rmi\lambda_0}}$ is guaranteed to be
positive definite, it must be true that the eigenspace $\E_{\rmi\lambda_0}$
resides in the positive space of the operator $\calA_0+\lambda_0^2\calA_2$.
In the applications we consider the operator $\calA_2$ will be positive
definite. In this case, if we again choose,
\[
S\coloneqq N(\calA_0)\oplus\ker(\calA_0),
\]
then the operator,
\[
P_{S^\perp}\left(\calA_0+\lambda_0^2\calA_2\right)P_{S^\perp}=
P_{S^\perp}\calA_0P_{S^\perp}+\lambda_0^2P_{S^\perp}\calA_2P_{S^\perp},
\]
will be positive definite. Consequently, if $\rmi\lambda_0$ is a polynomial
eigenvalue whose associated eigenfunction resides in $S^\perp$, then the
negative Krein index will be zero. %Note that unless $\calA_2$ is positive
%definite, it will generically be the case that the operators $P_S$ and
%$P_S^\perp$ will \textit{not} be spectral projections.
%
%\begin{remark}
%In practice it is better to choose the slightly larger space,
%\[
%S\coloneqq\left[N(\calA_0)\oplus\ker(\calA_0)\right]+N(\calA_2).
%\]
%The operator $P_{S^\perp}\calA_0 P_{S^\perp}$ will then be positive definite,
%which will lead to better properties of the Krein matrix for numerical
%computations.
%\end{remark}

\section{The Krein eigenvalues}\label{s:3}

Since $P_n(\lambda)$ is a star-even polynomial operator, the Krein matrix is a self-adjoint meromorphic family of operators in the spectral parameter, $\lambda$. In particular, the Krein matrix is Hermitian for purely imaginary $\lambda$.
Henceforth, write $\lambda=\rmi z$ for
$z\in\R$, and write the Krein matrix as
\[
\vK_S(z)=\calP_n(\rmi z)|_S-%
\calP_n(\rmi z)P_{S^\perp}(P_{S^\perp}\calP_n(\rmi z)P_{S^\perp})^{-1}P_{S^\perp}\calP_n(\rmi z)|_{S}.
\]
Since the Krein matrix is Hermitian for real $z$, for each value of
$z$ there are $n_S$ real-value eigenvalues, $r_j(z)$. These eigenvalues of
the Krein matrix are called the \textit{Krein eigenvalues}. The Krein eigenvalues are real meromorphic, as are the associated spectral projections. In particular, if the Krein eigenvalues are simple, the associated eigenvectors are real meromorphic. See Kato \cite[Chapter~VII.3]{kato:ptf80} for the details.

Since
\[
\det\vK_S(z)=\prod_{j=1}^{n_S}r_j(z),
\]
finding the zeros of the determinant of the Krein matrix is equivalent to
finding the zero set of each of the Krein eigenvalues.
One of the most important properties of the Krein eigenvalues is that the
sign of the derivative at a simple zero is related to the Krein index of that
polynomial eigenvalue. In order to see this, we start with
\begin{equation}\label{e:31a}
\vK_S(z)\vv_j(z)=r_j(z)\vv_j(z)\quad\leadsto\quad
r_j'(z)=\frac{\vv_j(z)^\rma\vK_s'(z)\vv_j(z)}{|\vv_j(z)|^2}.
\end{equation}
The latter equality is a solvability condition which follows upon noting that
both the Krein eigenvalue and its associated eigenvector are meromorphic and
consequently have convergent Taylor expansions. If $r_j(z)=0$, then the
components of the associated eigenvector correspond to the various basis
elements in the subspace $S$; namely, the associated eigenfunction is given
by
\begin{equation}\label{e:31}
\psi=\sum_{k=1}^{n_S}v^j_ks_k+s^\perp,\quad
\vv_j=\left(\begin{array}{c}v^j_1\\v^j_2\\\vdots\\v^j_{n_S}\end{array}\right),
\end{equation}
where the element $s^\perp$ is determined via \cref{e:pp3},
\[
s^\perp=-\sum_{k=1}^{n_S}v^j_k\left(P_{S^\perp}P_n(\rmi z)P_{S^\perp}\right)^{-1}P_{S^\perp}P_n(\rmi z)s_k.
\]

We now compute $\vK'(z)$. For the first term in the Krein matrix,
\[
\frac{\rmd}{\rmd z}\langle s_i,\calP_n(\rmi z)s_j\rangle=
\langle s_i,[\rmi\calP'_n(\rmi z)]s_j\rangle.
\]
The operator $\rmi\calP_n'(\rmi z)$ is Hermitian. Differentiating the second term requires
repeated applications of the product rule, as well as using the fact that the
operator $\calP_n(\rmi z)$ is Hermitian. Since
\[
\frac{\rmd}{\rmd z}(P_{S^\perp}\calP_n(\rmi z)P_{S^\perp})^{-1}=
-(P_{S^\perp}\calP_n(\rmi z)P_{S^\perp})^{-1}[\rmi\calP_n'(\rmi z)]
(P_{S^\perp}\calP_n(\rmi z)P_{S^\perp})^{-1},
\]
upon some simplification we can write
\[
\begin{aligned}
&\frac{\rmd}{\rmd z}\langle s_i,P_{S^\perp}\calP_n(\rmi z)(P_{S^\perp}\calP_n(\rmi z)P_{S^\perp})^{-1}P_{S^\perp}\calP_n(\rmi z)s_j\rangle=\\
&\quad\langle s_i,[\rmi\calP'_n(\rmi z)]P_{S^\perp}(P_{S^\perp}\calP_n(\rmi z)P_{S^\perp})^{-1}P_{S^\perp}\calP_n(\rmi z)s_j\rangle\\
&\qquad+\langle s_i,\calP_n(\rmi z)P_{S^\perp}(P_{S^\perp}\calP_n(\rmi z)P_{S^\perp})^{-1}P_{S^\perp}[\rmi\calP'_n(\rmi z)]s_j\rangle\\
&\qquad\quad-\langle s_i,\calP_n(\rmi z)P_{S^\perp}(P_{S^\perp}\calP_n(\rmi z)P_{S^\perp})^{-1}[\rmi\calP_n'(\rmi z)]
(P_{S^\perp}\calP_n(\rmi z)P_{S^\perp})^{-1}P_{S^\perp}\calP_n(\rmi z)s_j\rangle.
\end{aligned}
\]
The right-hand side has the compact form
\[
\frac{\rmd}{\rmd z}\langle s_i,P_{S^\perp}\calP_n(\rmi z)(P_{S^\perp}\calP_n(\rmi z)P_{S^\perp})^{-1}P_{S^\perp}\calP_n(\rmi z)s_j\rangle=
\langle s_i,(\calR+\calR^\rma)s_j\rangle-\langle s_i,\calS s_j\rangle,
\]
where
\begin{equation*}
\begin{aligned}
\calR&\coloneqq\calP_n(\rmi z)P_{S^\perp}(P_{S^\perp}\calP_n(\rmi z)P_{S^\perp})^{-1}P_{S^\perp}[\rmi\calP'_n(\rmi z)]\\
\calS&\coloneqq\calP_n(\rmi z)P_{S^\perp}(P_{S^\perp}\calP_n(\rmi z)P_{S^\perp})^{-1}[\rmi\calP_n'(\rmi z)]
(P_{S^\perp}\calP_n(\rmi z)P_{S^\perp})^{-1}P_{S^\perp}\calP_n(\rmi z).
\end{aligned}
\end{equation*}
In conclusion, the derivative
of the Krein matrix is
\begin{equation}\label{e:32}
\vK'(z)=[\rmi\calP_n'(\rmi z)]_S+\calS|_S-(\calR+\calR^\rma)|_S,
\end{equation}
where the operators $\calR,\calS$ are defined above.

We now compute the Krein index using our decomposition of an eigenfunction.
For the sake of exposition, let us assume that the polynomial eigenvalue is
simple. Using the decomposition \cref{e:31} with $\vK_s(z)\vv_j(z)=\vn$, we
have
\[
[\rmi\calP_n'(\rmi z)]\psi=\sum_{k=1}^{n_S}v_k^j[\rmi\calP_n'(\rmi z)]s_k-
\sum_{k=1}^{n_S}v_k^j[\rmi\calP_n'(\rmi z)]
\left(P_{S^\perp}P_n(\rmi z)P_{S^\perp}\right)^{-1}P_{S^\perp}P_n(\rmi z)s_k.
\]
Upon taking the inner product with $\psi$, and using the fact that
$\calP_n(\rmi z)$ is Hermitian,
\[
\langle\psi,[\rmi\calP_n'(\rmi z)]\psi\rangle=
\vv_j(z)^\rma\left([\rmi\calP_n'(\rmi z)]|_S+\calS|_S-(\calR+\calR^\rma)|_S\right)\vv_j(z).
\]
Upon comparing with \cref{e:32} we conclude
\[
\langle\psi,[\rmi\calP_n'(\rmi z)]\psi\rangle=\vv_j(z)^\rma\vK'(z)\vv_j(z),
\]
where the eigenfunction $\psi$ has the expansion provided for in \cref{e:31}.

Going back to \cref{e:31a}, we have that the derivative of the Krein
eigenvalue can be expressed in terms of the eigenfunction as
\[
r_j'(z)=\frac{\langle\psi,[\rmi\calP_n'(\rmi z)]\psi\rangle}{|\vv_j(z)|^2}.
\]
Going further back to the
definition of the negative Krein index, we can conclude the desired result.
If $\rmi z$ is a polynomial eigenvalue with $r_j(z)=0$, then the Krein index
is related through the derivative via
\[
k_\rmi^-(\rmi z)=\begin{cases}0,\quad&zr_j'(z)<0\\1,\quad&zr_j'(z)>0.\end{cases}
\]

Since our goal is to quickly and easily read off the Krein signature via a
graph of the Krein eigenvalues, we will redefine the Krein matrix as
\[
\vK_S(z)=-z\left[\calP_n(\rmi z)|_S-%
\calP_n(\rmi z)P_{S^\perp}(P_{S^\perp}\calP_n(\rmi z)P_{S^\perp})^{-1}P_{S^\perp}\calP_n(\rmi z)|_{S}\right].
\]
This redefinition adds a singularity to the Krein matrix at $z=0$, but in the search for nonzero polynomial eigenvalues this is an unimportant consequence. On the other hand, the Krein eigenvalues for the new matrix are related to the original matrix
via $r_j(z)\mapsto -zr_j(z)$.
Thus, at a zero of the Krein eigenvalue we have
the mapping $r_j'(z)\mapsto -zr_j'(z)$, so for the new Krein matrix we have
the relationship
\[
k_\rmi^-(\rmi z)=\begin{cases}0,\quad&r_j'(z)>0\\1,\quad&r_j'(z)<0.\end{cases}
\]
A positive slope of a Krein eigenvalue at a zero corresponds to a polynomial
eigenvalue with positive signature, whereas a negative slope shows that the
polynomial eigenvalue has negative Krein signature.

If the zero of a Krein eigenvalue is not simple, then the corresponding
polynomial eigenvalue has a Jordan chain, and the negative Krein index
depends upon the length of the chain, see \cite[Section~2.2]{kapitula:tks10} and the references therein. For example, if $r_j(z)=r_j'(z)=0$ with $r_j''(z)\neq0$, then there will be a Jordan chain of length two; moreover, the negative Krein index associated with the Jordan chain will be one. In general, a zero of order $m$ implies a Jordan chain of length $m$, and the negative Krein index associated with that chain will be roughly half the length of the chain. We
will not provide any more details here, as in our examples the polynomial eigenvalues will be simple. In summary,
we have the following result:

\begin{theorem}\label{thm:krein}
For $n\in\{1,2\}$ consider the star-even polynomial,
\[
\calP_n(\lambda)=\sum_{j=0}^n\lambda^j\calA_j,
\]
which acts on a Hilbert space, $X$, with inner-product,
$\langle\cdot,\cdot\rangle$.  Suppose $\calA_0$ has compact resolvent. Set $P_{\calA_0}:X\mapsto\ker(\calA_0)$ to be the spectral projection onto the kernel, and
    $P_{\calA_0}^\perp=\calI-P_{\calA_0}$. Further suppose the operator coefficients satisfy,
\begin{enumerate}
\item $\rmn(\calA_0)$ is finite
\item for $j=1,2$ the
    operators,
\[
\left(P_{\calA_0}^\perp\calA_0P_{\calA_0}^\perp\right)^{-1}
P_{\calA_0}^\perp\calA_jP_{\calA_0}^\perp:\ker(\calA_0)^\perp\mapsto\ker(\calA_0)^\perp,
\]
are compact.
\end{enumerate}
Regarding the Krein matrix, first let $S\subset X$ be a given finite-dimensional subspace, and $P_{S^\perp}:X\mapsto S^\perp$ be the orthogonal
projection. The Krein matrix associated with $S$ is,
\[
\vK_S(z)=-z\left[\calP_n(\rmi z)|_S-%
\calP_n(\rmi z)P_{S^\perp}(P_{S^\perp}\calP_n(\rmi z)P_{S^\perp})^{-1}P_{S^\perp}\calP_n(\rmi z)|_{S}\right].
\]
The Krein eigenvalues, $r_j(z)$ for $j=1,\dots,\dim[S]$, are the eigenvalues of the Krein matrix. If $z\in\R$, the Krein eigenvalues are meromorphic. Moreover, if $\lambda=\rmi z$ is a polynomial
eigenvalue with $\calP_n(\rmi z)\psi=0$,
\begin{enumerate}
\item then either $r_j(z)=0$ for at least one $j$, or $\psi\in S^\perp$
\item if $z\in\R$, and if $r_j(z)=0$ for some $j$, then the Krein
    signature of a semi-simple polynomial eigenvalue is determined by the slope of the graph of the Krein eigenvalue,
\[
k_\rmi^-(\rmi z)=\begin{cases}0,\quad&r_j'(z)>0\\1,\quad&r_j'(z)<0.\end{cases}
\]
\end{enumerate}
\end{theorem}

\begin{remark}
Recall that the choice,
\[
S=N(\calA_0)\oplus\ker(\calA_0),
\]
ensures that all polynomial eigenvalues with negative Krein signature are seen as zeros of one or more Krein eigenvalues.
\end{remark}
%
%\subsection{First-order operator polynomials}
%
%\subsection{Second-order operator polynomials}

In its general form the Krein matrix looks to be complicated, and does not appear to have an underlying intuitively understood structure.  However, as we shall see in our subsequent examples, the Krein matrix can have intimate connections with dispersion relations, the Hale-Sandstede-Lin's method for constructing multi-pulses, etc.

\section{First application: modulational instabilities for small amplitude periodic solutions}\label{s:4}

For our first application we show how the Krein matrix can be used to understand the existence of
instability bubbles, i.e., a curve of unstable spectra which is attached to the imaginary axis, for small spatially periodic waves to dispersive systems.
The instabilities will not necessarily be associated with high-frequency
(long wavelength) perturbations. Without loss of generality we will assume the spatial period is $2\pi$.

Regarding the existence problem we will assume it is of the form,
\begin{equation}\label{e:51}
\calL u-cu+f(u)=0,
\end{equation}
where
\begin{enumerate}
\item $\displaystyle{\calL=\sum_{j=0}^Na_j\ell^{2j}\partial_x^{2j}}$ with
    $\ell,(-1)^Na_{2N}>0$
\item $c\in\R$ is a free parameter (typically the wavespeed)
\item $f(u)$ is a smooth nonlinearity with $f(0)=f'(0)=0$.
\end{enumerate}
The parameter $\ell$ can be adjusted via a rescaling of $x$.
The operator $\calL$ is self-adjoint under the inner-product,
\[
\langle f,g\rangle=\int_0^{2\pi}f(x)\overline{g(x)}\,\rmd x.
\]

\begin{remark}
The nonlinearity could be more general, $f=f(u,\partial_xu,\dots)$. All that is required is that it be smooth and (at least) quadratic in the arguments near the origin, and that it be unchanged under reversibility, $x\mapsto-x$.
\end{remark}

We briefly sketch the argument leading to the existence of a family of small spatially periodic solutions. The details can be found in \cite[Theorem~3.15]{haragus:lbc11}.
The characteristic polynomial associated with the ordinary differential operator
$\calL$ is
\[
p_{\calL}(r,\ell)=\sum_{j=0}^Na_j\ell^{2j}r^{2j}.
\]
Regarding the characteristic polynomial we assume there is an $\ell_0$ such that,
\begin{enumerate}
\item $\partial_rp_{\calL}(\rmi,\ell_0)\neq0$
\item upon setting the zero amplitude wavespeed,
\begin{equation}\label{e:condb}
c_0\coloneqq p_{\calL}(\rmi,\ell_0)=\sum_{j=0}^N(-1)^ja_j\ell_0^{2j},
\end{equation}
there is no positive real $k\neq1$ such that $p_{\calL}(\rmi
    k,\ell_0)-c_0=0$.
\end{enumerate}
There will then exist a family of
$2\pi$-periodic solutions, say $U(x)$, with the properties:
\begin{enumerate}
\item $U(x)=U(-x)$
\item $U(x)=\epsilon A\cos(x)+\calO(\epsilon^2)$ for $A>0$
\item $\ell=\ell_0+\calO(\epsilon)$ (the $\calO(\epsilon)$ terms depend on
    $A$).
\end{enumerate}

If \cref{e:condb} above does not hold, i.e., if there are other purely
imaginary roots to $p_{\calL}(r,\ell_0)-\beta_0=0$, then the equations on the
center-manifold will still be reversible. However, the dimension of the
manifold (equal to the number of purely imaginary roots, counting
multiplicity) increases, and since the reduced system is no longer planar it
is not clear if there are still periodic (versus quasi-periodic) solutions.
The case of a second additional imaginary root, $\pm\rmi q$ with $q>1$, is
discussed by \cite[Chapter~4.3.4]{haragus:lbc11}. If $q$ is irrational, or
if $q\ge5$, only KAM tori are expected, and consequently only quasi-periodic
solutions. In the case of strong resonance, $q=2$, the equations on the
center manifold are completely integrable, and there can be periodic orbits,
homoclinic orbits, and orbits homoclinic to periodic orbits. The other
resonant case of $q=3$ is still open.
In conclusion, we can safely assume the existence of small $2\pi$-periodic
solutions to \cref{e:51}.

% \subsection{First-order in time}

We now consider the spectral stability of these spatially periodic solutions.
Consider the KdV-like and first-order-in-time Hamiltonian system,
\begin{equation}\label{e:52}
\partial_t u+\partial_x\left(\calL u+f(u)\right)=0.
\end{equation}
The nonlinearity $f(u)$ satisfies the assumption (c) above, while
\[
\calL u=\sum_{j=0}^Na_j\partial_x^{2j}u,\quad (-1)^Na_{2N}>0.
\]
In traveling coordinates, $z\coloneqq x-ct$, the equation becomes,
\[
\partial_tu+\partial_z\left(\calL u-cu+f(u)\right)=0,\quad
\partial_x^{2j}\mapsto\partial_z^{2j}.
\]
Upon rescaling of time and space,
\[
\tau=\ell t,\quad y=\ell z,
\]
we have the PDE to be studied,
\begin{equation}\label{e:53}
\partial_\tau u+\partial_y\left(\calL u-cu+f(u)\right)=0,
\end{equation}
where
\[
\calL u=\sum_{j=0}^Na_j\ell^{2j}\partial_x^{2j}u,\quad (-1)^Na_{2N}>0.
\]
Following the previous discussion, upon setting,
\[
c_0\coloneqq p_{\calL}(\rmi,\ell_0),
\]
where $\ell_0$ is chosen so that $p_{\calL}(\rmi k,\ell_0)-c_0=0$ has no integral solutions for $k>1$, we know there is a family of small $2\pi$-periodic solutions,
$U(x)=\calO(\epsilon)$, for $0<\epsilon\ll1$.

We now consider the spectral stability of such solutions. The linearized
problem is,
\[
\partial_\tau v+\partial_y\left(\calL v-c_0v+f'(U)v\right)=0,\quad |f'(U)|=\calO(\epsilon).
\]
Using separation of variables, $v(y,\tau)=\rme^{\lambda\tau}v(y)$, we arrive
at the spectral problem,
\begin{equation}\label{e:54}
\lambda v+\partial_y\left(\calL v-c_0v+f'(U)v\right)=0,\quad |f'(U)|=\calO(\epsilon).
\end{equation}
We use a Bloch decomposition to understand the spectral problem, see
\cite[Chapter~3.3]{kapitula:sad13}. Writing for $-1/2<\mu\le1/2$,
\[
v(y)=\rme^{\rmi\mu y}w(y),\quad w(y+2\pi)=w(y),
\]
the problem \cref{e:54} becomes,
\begin{equation}\label{e:55}
\lambda w+(\partial_y+\rmi \mu)\left(\calL_\mu w-c_0w+f'(U)w\right)=0,\quad |f'(U)|=\calO(\epsilon),
\end{equation}
where
\[
\calL_\mu=\sum_{j=0}^Na_{2j}\ell_0^{2j}(\partial_y+\rmi \mu)^{2j}.
\]
Because
the underlying wave is even in $x$, it is sufficient to consider
$0\le\mu\le1/2$; in particular, if
$\lambda$ is an eigenvalue associated with $\mu$, then $\overline{\lambda}$
is an eigenvalue associated with $-\mu$, see \cite[Section~4]{haragus:ots08}. For fixed $\mu$ the spectrum will be
discrete, countable, and have an accumulation point only at $\infty$. The
full spectrum, which is essential spectra only, will be the union of all the
point spectra as $\mu$ is varied over the range.

We are henceforth interested only in sideband instabilities, $\mu>0$.
Set,
\[
\calA_0\coloneqq\calL_\mu-c_0+f'(U).
\]
The operator $\calA_0$ is self-adjoint on the space of $2\pi$-periodic
functions endowed with the natural $L^2[0,2\pi]$ inner product.
The invertible operator $\partial_y+\rmi\mu$ is skew-Hermitian.  Since $\calA_0$ is self-adjoint with smooth dependence on parameters, each of the eigenvalues of $\calA_0$ is smooth in $(\mu,\epsilon)$ \cite{kato:ptf80}. The same can be said of the composition, $(\partial_y+\rmi\mu)\calA_0$, except at possibly the finite number of points where there are Jordan chains.
Consequently, we will first consider the spectral problem
when $\epsilon=0$. Afterwards, we will make generic statements about what
will happen for $\epsilon>0$ small.

For $0<\mu\le1/2$ we rewrite the spectral problem in the star-even form,
\begin{equation}\label{e:56}
\calA_0w+\lambda\calA_1w=0,\quad
\calA_1\coloneqq\left(\partial_y+\rmi\mu\right)^{-1}.
\end{equation}
The boundary conditions associated with this problem are periodic,
$w(y+2\pi)=w(y)$.
First assume $\epsilon=0$, so that $f'(U)\equiv0$.
The spectrum for \cref{e:56} is straightforward to compute using a Fourier analysis. Letting
$w(y)=\rme^{\rmi ny}$ for $n\in\Z$ we get a sequence of problems,
\begin{equation}\label{e:57}
d(n,\mu)+\lambda\frac{1}{\rmi(n+\mu)}=0,
\end{equation}
where the first term is the dispersion relation associated with the
steady-state problem,
\[
d(n,\mu)\coloneqq\sum_{j=0}^N(-1)^ja_{2j}\ell_0^{2j}(n+\mu)^{2j}-c_0.
\]
We first show that the spectrum of $\calA_0$ has a nonzero and finite number of negative eigenvalues for at least some values of $\mu$.
First suppose $\epsilon=0$.
The existence assumption implies $d(\pm1,0)=0$. For small $\mu$ we have the
expansions,
\begin{equation}\label{e:58}
d(\pm1,\mu)=\pm\left(2\sum_{j=1}^N(-1)^jja_j\right)\mu+\calO(\mu^2).
\end{equation}
Consequently, $d(+1,\mu)d(-1,\mu)<0$ for small $\mu$, so one of $d(\pm1,\mu)$ is negative for small $\mu$. Consequently, $\rmn(\calA_0)\ge1$. The assumption $(-1)^Na_{2N}>0$ implies there
is an $N_0$ such that $d(n,\mu)>0$ for $|n|\ge N_0$. Consequently, there can be at most a finite number of negative eigenvalues, so
$\rmn(\calA_0)<\infty$. By continuity $\rmn(\calA_0)$ will remain unchanged for $\epsilon>0$ and small.

We now construct the Krein matrix, and then use it to analyze the spectrum.
Assume there is a sequence
$n_1,n_2,\dots,n_q$ such that $d(n,\mu)<0$ for $n\in\{n_1,n_2,\dots,n_q\}$,
and $d(n,\mu)>0$ for $n\notin\{n_1,n_2,\dots,n_q\}$. Clearly,
$\rmn(\calA_0)=q$. We take as our space $S=N(\calA_0)$,
\[
S=\Span\{\rme^{\rmi n_1y},\rme^{\rmi n_2 y},\dots,\rme^{\rmi n_q y}\}.
\]
Since
\[
P_1(\rmi z)S=S,\quad P_1(\rmi z)S^\perp=S^\perp,
\]
the Krein matrix as described in \cref{thm:krein} collapses to
\[
\begin{aligned}
\vK_S(z)&=-z\calP_1(\rmi z)|_S\\
&=-z\diag\left(d(n_1,\mu)+\frac{z}{n_1+\mu},\dots,d(n_q,\mu)+\frac{z}{n_q+\mu}\right).
\end{aligned}
\]
The expected poles, which are the eigenvalues of the sandwiched operator,
\[
P_{S^\perp}\calP_1(\rmi z)P_{S^\perp}=P_1(\rmi z)S^\perp,
\]
are located at
$z_n^\rmp=-(n+\mu)d(n,\mu)$ for $n\notin\{n_1,n_2,\dots,n_q\}$, and are
removable singularities. All of the poles are polynomial eigenvalues for the spectral
problem. Since they correspond to removable singularities, the polynomial
eigenvalues all have positive Krein signature.

\begin{remark}
The poles are removable when $\epsilon=0$ because  $[\calP_1(\lambda)S]\cap S^\perp=\{0\}$. In particular, it follows from the fact that the $\epsilon=0$ problem has constant coefficients.  One expects that for $\epsilon>0,\,[\calP_1(\lambda)S]\cap S^\perp$ has a nontrivial intersection. Thus, the expectation is that the poles will no longer be removable for small amplitude waves.
\end{remark}

The Krein eigenvalues are
\[
r_j(z)=-z\left(d(n_j,\mu)+\frac{z}{n_j+\mu}\right),\quad j=1,\dots,q.
\]
The nonzero zeros of the Krein eigenvalues,
\[
z_j^\rmn=-(n_j+\mu)d(n_j,\mu),\quad j=1,\dots,q,
\]
satisfy
\[
r_j'(z_j^\rmn)=d(n_j,\mu)<0,
\]
so these zeros correspond to polynomial eigenvalues with negative Krein
signature. In conclusion, via Fourier analysis we have located all of the polynomial eigenvalues, and through the Krein eigenvalues we have identified those which have a negative Krein index.

\begin{remark}
Note that for constant states, $\epsilon=0$, the Krein signature can be directly computed from the dispersion relation.  For fixed $\mu$ the Krein
eigenvalues are dispersion curves that correspond to polynomial eigenvalues
with negative Krein index, and the poles correspond to dispersion curves with
positive Krein index. If the two curves intersect, then there is a collision
of polynomial eigenvalues with opposite Krein signature. Consequently, for a small amplitude wave the intersection of a Krein eigenvalue with a (potentially) removable singularity of the Krein matrix can be noted
without actually computing a Krein eigenvalue. This graphical
approach towards spectral stability by looking at the dispersion curves is the one taken by \cite{deconinck:hfi16,kollar:dco19,trichtchenko:sop18}. The Krein matrix approach is more robust in the sense that while it, too, is graphical in nature, it does not necessarily assume that the underlying waves have small amplitude. In particular, the smallness assumption allows for an analytic construction of the matrix; however, if the wave has an $\calO(1)$ amplitude, then the Krein matrix can still be constructed numerically, and the graphical analysis will still hold for this numerically constructed matrix.
\end{remark}

When $\epsilon=0$ the wave is spectrally stable, and all of the spectra is purely imaginary. For $\epsilon>0$ a spectral instability can arise for the small amplitude wave only through the collision of a purely imaginary
polynomial eigenvalue with positive Krein index and one with negative Krein
index. This collision generically leads to a Hamiltonian-Hopf bifurcation,
see \cite[Chapter~7.1.2]{kapitula:sad13} and the references therein. If for a fixed $\mu_0$ there is a polynomial eigenvalue
with positive real part, then such polynomial eigenvalues will exist for $\mu$ in a
neighborhood of $\mu_0$. If for $\mu_0$ the polynomial eigenvalue with positive real part is simple, then the union of all polynomial eigenvalues for $\mu$ in a neighborhood of $\mu_0$ will form a smooth curve. We will call this curve an instability bubble. In our example any instability bubble will have an $\calO(1)$
imaginary part; consequently, they will not be related to instability curves coming from the origin which arise due to a long wavelength modulational instability.
A bubble intersects the imaginary axis, and because of the
$\{\lambda,-\overline{\lambda}\}$ reflection symmetry about the imaginary
axis, the curve on the left of the imaginary axis is a mirror image of that
on the right.

The Krein eigenvalues reflect this collision of polynomial eigenvalues with opposite index in one of two possible ways. The
first is that a Krein eigenvalue has a double zero at the time of collision,
see \cite[Lemma~2.8]{kapitula:tks10}. For small waves this cannot happen, as
the explicit form of the Krein eigenvalues shows that all of the zeros are
simple for the limiting zero amplitude wave.

As for the other possible collision scenario, recall that when $\epsilon=0$ a
zero of a Krein eigenvalue corresponds to a polynomial eigenvalue with
negative Krein signature, while all the removable singularities, i.e.,
polynomial eigenvalues of the operator $P_{S^\perp}P_1(\rmi z)P_{S^\perp}$,
correspond to polynomial eigenvalues with positive Krein signature. If a
simple zero is isolated, then the Krein matrix being meromorphic implies via
a winding number calculation that the zero remains simple for small
perturbations. Moreover, the spectral symmetry implies the polynomial
eigenvalue must remain purely imaginary. Now, suppose that a simple zero
coincides with a simple removable singularity, so when $\epsilon=0$ the
winding number is again one. For the problem at hand this situation is
realized when a zero of one of the Krein eigenvalues intersects one of the
removable singularities, $z_n^\rmp$. In general, this intersection must be
computed numerically. Assume that upon perturbation the singularity is no
longer removable - it will remain simple. In this case the invariance of the
winding number to small perturbation implies there must now be two zeros. The
spectral symmetry implies these correspond to either two purely imaginary
polynomial eigenvalues, or a pair of polynomial eigenvalues with nonzero real
part. In the former case, the invariance of the HKI to small perturbation
implies that one polynomial eigenvalue will have positive Krein signature,
whereas the other will have negative Krein signature. The latter case
corresponds to the onset of a Hamiltonian-Hopf bifurcation. An analytic
argument which leads to the same conclusion is presented in
\cite[Section~2.4]{kapitula:tks10}.

In conclusion, the total number of bubbles that can form is bounded above by
the number of intersections of Krein eigenvalues with poles. Supposing that
the HKI is fixed for all $\mu$, this leaves open the possibility that the
number of bubbles is greater than $K_{\Ham}$.  For example, suppose
$K_{\Ham}=2$, so that for each $\mu$ there can be at most two polynomial
eigenvalues with positive real part. Since there will be two Krein
eigenvalues, for each $\mu$ there can be at most two associated bubbles.
However, overall there can be more than two bubbles. Suppose there is a
sequence $0<\mu_1<\mu_2<\cdots<\mu_N$ for which a Krein eigenvalue intersects
a pole. A Hamiltonian-Hopf bifurcation is then possible for $\mu$ near each
$\mu_j$, which leaves open the possibility of having up to $N$ bubbles.

\begin{remark}
More generally, if $k$ polynomial eigenvalues with negative signature
coincide with a removable singularity for the Krein matrix of order $\ell$,
then upon perturbation the invariance of the winding number implies that
$k+\ell$ polynomial eigenvalues will be created via the collision. The
invariance of the HKI implies that $k=k_\rmc+k_\rmi^-$, where here $k_\rmi^-$
corresponds to the number of purely imaginary polynomial eigenvalues with
negative Krein signature which are close to the unperturbed eigenvalue, and
$k_\rmc$ is the number of polynomial eigenvalues with positive real part
which are close to the unperturbed eigenvalue. As for the number of
polynomial eigenvalues associated with the order of the removable
singularity, $\ell=k_\rmc+k_\rmi^+$, where  here $k_\rmi^+$ corresponds to
the number of purely imaginary polynomial eigenvalues with positive Krein
signature which are close to the unperturbed eigenvalue.
\end{remark}

\begin{figure}[ht]%[tbp]
\begin{center}
\includegraphics{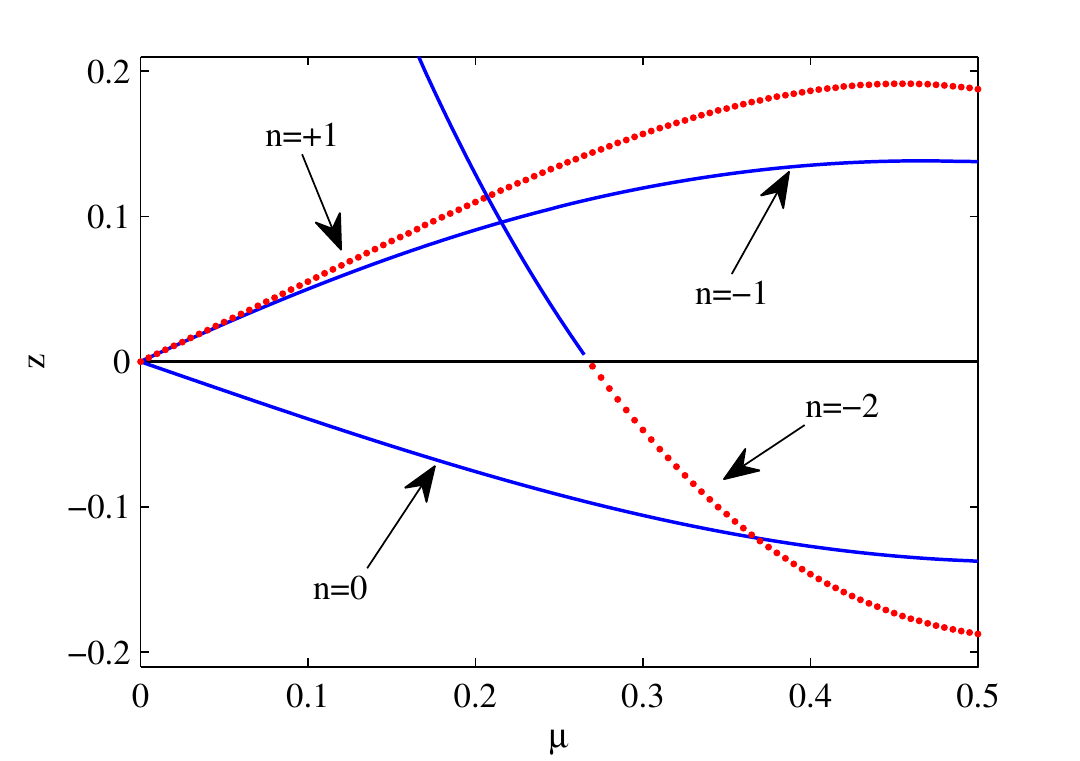}
\caption{(color online) Plots of the dispersion relations, $z_n(\mu)$, for the linearization of
\cref{e:502} for relevant values of $n$ when $b=-8/15$. A dotted curve corresponds to an eigenvalue with negative Krein index,
while a solid curve shows an eigenvalue with positive index. Not only is a Hamiltonian-Hopf
bifurcation possible for small $\mu$, it is possible for $\mu\sim0.21$ and $\mu\sim0.37$.}
\label{f:FifthOrderKdVSpectra}
\end{center}
\end{figure}

For a particular example, consider the fifth-order KdV-like equation,
\begin{equation}\label{e:501}
\partial_tu+\partial_x\left(\frac2{15}\partial_x^4u-b\partial_x^2u+\frac32u^2
+\frac12[\partial_xu]^2+u\partial_x^2u\right)=0.
\end{equation}
This weakly nonlinear long-wave equation arises as an approximation to the
classical gravity-capillary water-wave problem \cite{champneys:agi97}. Here
$u(x,t)$ is the surface elevation with respect to the underlying normal water
height, and $b\in\R$ is the offset of the Bond number (a measure of surface
tension) from the value $1/3$.
In traveling coordinates, $z=x-ct$, the
equation \cref{e:501} becomes
\begin{equation}\label{e:502}
\partial_tu+\partial_z\left(\frac2{15}\partial_z^4u-b\partial_z^2u-cu+\frac32u^2
+\frac12[\partial_zu]^2+u\partial_z^2u\right)=0.
\end{equation}
The wavespeed $c$ is a free parameter. To the best of our knowledge the spectral stability of small periodic waves to equation \cref{e:501} has not yet been studied.
However, the spectral stability of small spatially periodic waves to the Kawahara equation, which is \cref{e:501} with the last two terms in the open brackets removed, was recently studied by \cite{trichtchenko:sop18}.

\begin{figure}[ht]%[tbp]
\begin{center}
\begin{tabular}{cc}
\includegraphics{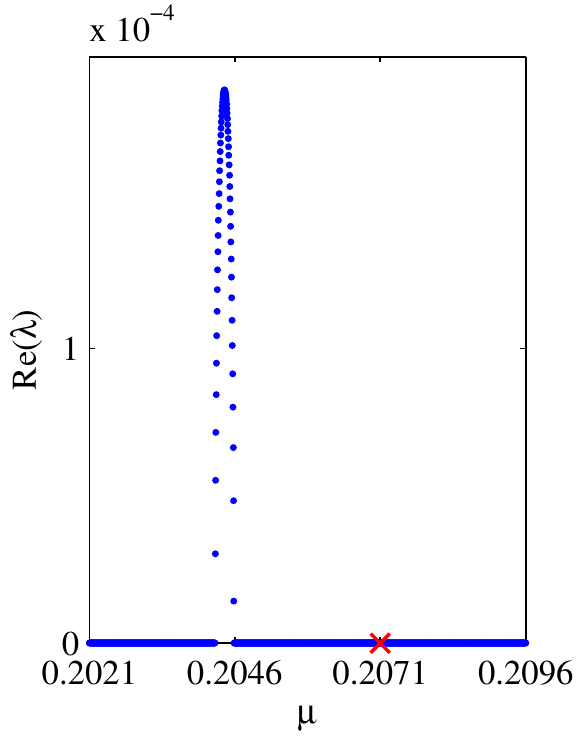} &
\includegraphics{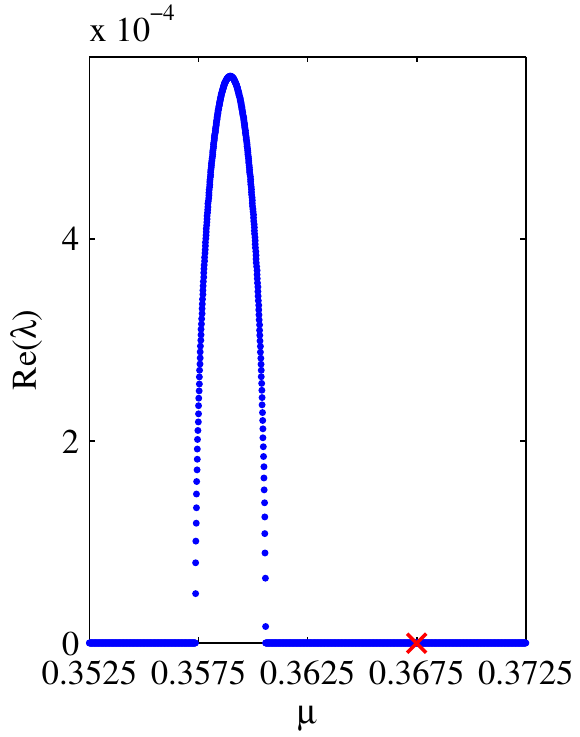}
\end{tabular}
\caption{(color online) Plots of the absolute value of the real part of the spectrum
for various values of $\mu$ for a wave with approximate amplitude $2.3\times10^{-2}$.
The plot on the left is for $\mu$ values near the zero/pole collision point $\mu\sim0.207$,
and the plot on the right is for $\mu$ values near the zero/pole collision point $\mu\sim0.368$. The $\mu$ value
for which the collision occurs is marked by a (red) cross.}
\label{f:specamp}
\end{center}
\end{figure}

First consider the existence problem. As discussed by
\cite[Section~4]{sandstede:hfb13} (also see \cite{champneys:agi97}), the
fourth-ODE,
\[
\frac2{15}\partial_z^4u-b\partial_z^2u+\frac32u^2
+\frac12[\partial_zu]^2+u\partial_z^2u=0,
\]
is a reversible Hamiltonian system. The position and momentum variables
are
\[
q_1=u,\quad q_2=\partial_zu,\quad p_1=-\frac2{15}\partial_z^3u+b\partial_zu-u\partial_zu,\quad
p_2=\frac2{15}\partial_z^2u,
\]
and the (analytic) Hamiltonian is
\[
H=-\frac12q_1^3-\frac12cq_1^2+p_1q_1-\frac12bq_2^2+\frac{15}{4}p_2^2+\frac12q_1q_2^2.
\]
The symplectic matrix for the system is the canonical one. Setting
\[
c=c_0\coloneqq\frac{2}{15}+b,
\]
the eigenvalues for the linearization of this Hamiltonian system about the
origin satisfy
\[
r^2=-1,\quad r^2=1+\frac{15}{2}b.
\]
If $b>-2/15$, then the center-manifold is two-dimensional, and the existence
of a family of periodic orbits follows from reversibility. If $b<-2/15$, but
$b\neq-2(1+m^2)/15$ for $m=1,2,\dots$ (the non-resonance condition), then one
can invoke the Lyapunov center theorem to conclude the existence of a family
of small periodic orbits with period close to $2\pi$ (see
\cite{buzzi:rhl05,weinstein:nmf73} for a discussion). In either case, the
period can be fixed to be $2\pi$ via a rescaling of the spatial variable. We will assume for that sake of exposition that $b=-8/15$, so $c_0=-6/15$. For this value of $b$ the ODE system is not in resonance.

\begin{figure}[ht]%[tbp]
\begin{center}
\begin{tabular}{cc}
\includegraphics{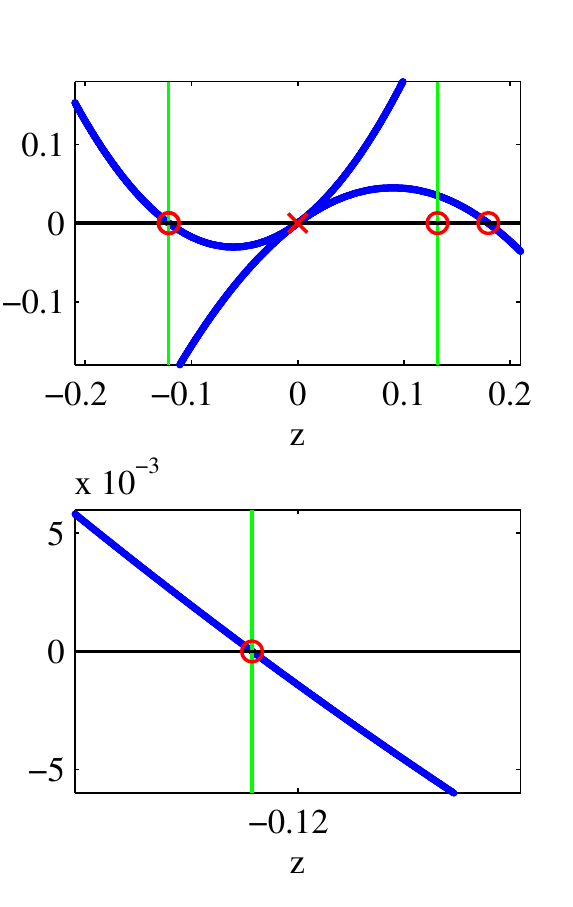} &
\includegraphics{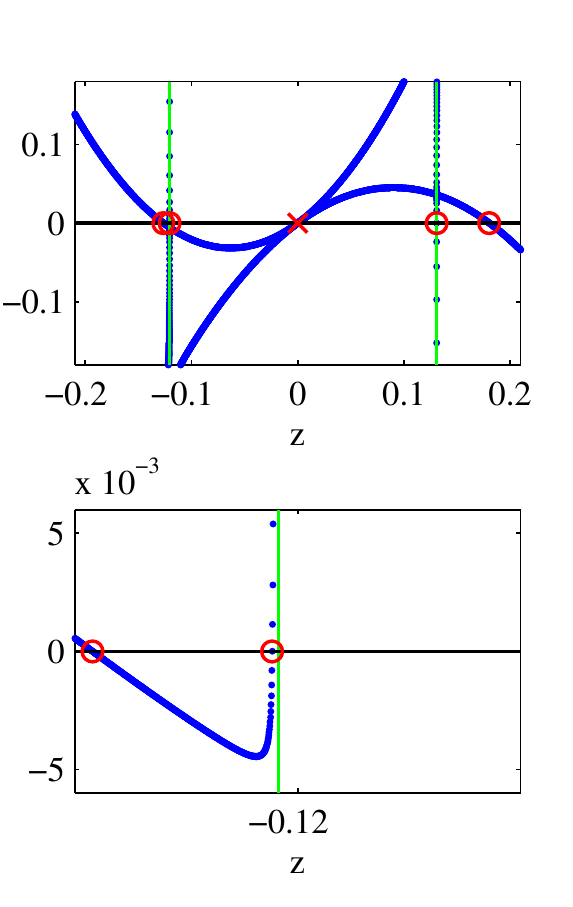}
\end{tabular}
\caption{(color online) Plots of the Krein eigenvalues for the trivial state (left figures) and for
a wave with approximate amplitude $2.3\times10^{-2}$ (right figures). The top two figures show the
situation at the zero/pole collision
point, $\mu\sim0.368$. The (red) circles correspond to polynomial eigenvalues, and the (red) cross
is the spurious zero of the Krein eigenvalues. The (green) vertical lines are poles
of the Krein matrix. In each quadrant the bottom figure is a blow-up
of the top figure near the polynomial eigenvalues of interest. Upon perturbation the zeros of the Krein eigenvalues remain purely real.}
\label{f:KreinEvalCollide1}
\end{center}
\end{figure}

We now consider the spectral stability of the periodic wave.
For the unperturbed problem the operator $\calA_0$ is,
\[
\calA_0=\frac2{15}(\partial_z+\rmi\mu)^4+\frac8{15}(\partial_z+\rmi\mu)^2+\frac6{15},
\]
so the dispersion relationship is,
\[
d(n,\mu)=\frac2{15}(n+\mu)^4-\frac8{15}(n+\mu)^2+\frac6{15}.
\]
It is straightforward to check that $d(n,\mu)>0$ for $\mu\notin\{-2,+1\}$.
Moreover, we have $d(+1,\mu)<0$, and
\[
d(-2,\mu)\begin{cases}>0,\quad&0<\mu<\mu_{\mathrm{ch}}\\
<0,\quad&\mu_{\mathrm{ch}}<\mu<1/2,\end{cases}
\]
where
\[
\mu_{\mathrm{ch}}\coloneqq2-\sqrt{3}\sim0.26795.
\]
Consequently,
\[
\rmn(\calA_0)=\begin{cases}1,\quad&0<\mu<\mu_{\mathrm{ch}}\\
2,\quad&\mu_{\mathrm{ch}}<\mu<1/2.\end{cases}
\]
Since the negative index of an invertible operator is unchanged for small perturbations, we know there is a $0<\mu_0\ll1$ such that if $\mu$ is in one of two intervals,
\[
\mu\in\left(\mu_0,\mu_{\mathrm{ch}}-\mu_0\right)\cup
\left(\mu_{\mathrm{ch}}+\mu_0,1/2\right),
\]
then $\rmn(\calA_0)$ remains unchanged for sufficiently small $\epsilon$. Going back to equation \cref{e:11}, we then know that for small $\epsilon$ the HKI is,
\[
k_\rmr+k_\rmc+k_\rmi^-=\begin{cases}1,\quad&\mu_0<\mu<\mu_{\mathrm{ch}}\\
2,\quad&\mu_{\mathrm{ch}}<\mu<1/2.\end{cases}
\]
If there are instability bubbles for the perturbed problem, there can be at most one for $\mu<\mu_{\mathrm{ch}}$, and at most two for $\mu_{\mathrm{ch}}<\mu<1/2$. For $0\le\mu<\mu_0$ a curve of unstable spectra may arise from the origin. We will not consider that here, but an example calculation for the KdV with general nonlinearity is provided in \cite[Section~4]{haragus:ots08}.

\begin{remark}
The transition point in the index, $\mu_{\mathrm{ch}}$, depends on $\epsilon$. For our purposes it is sufficient to consider how the number of instability bubbles depends on the change in $\rmn(\calA_0)$ between the two $\mu$-intervals without worrying about the precise boundary between the intervals.
\end{remark}

\begin{figure}[ht]%[tbp]
\begin{center}
\begin{tabular}{cc}
\includegraphics{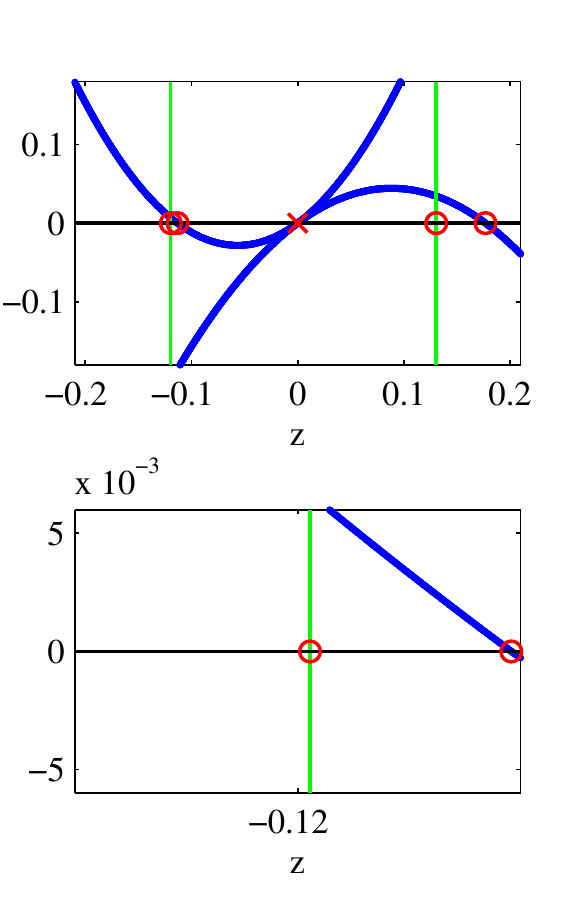} &
\includegraphics{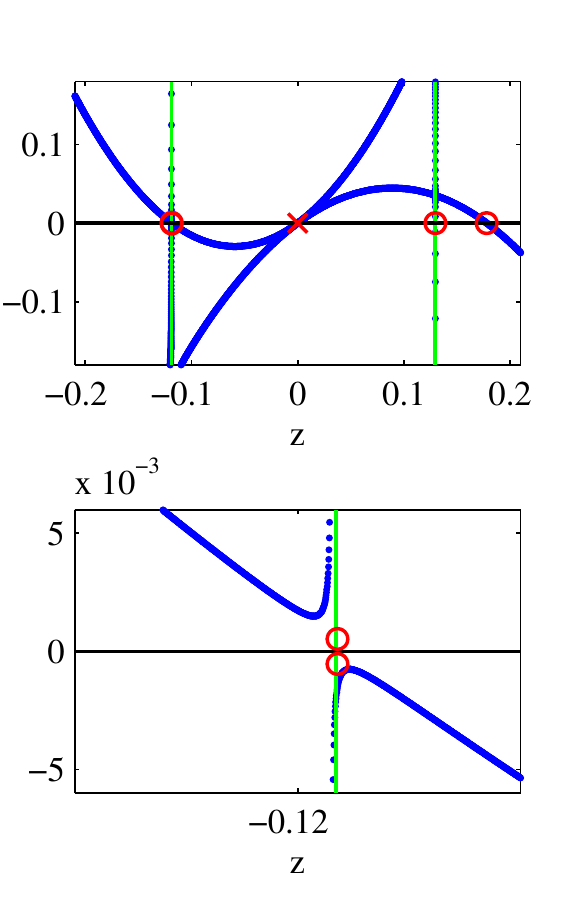}
\end{tabular}
\caption{(color online) Plots of the Krein eigenvalues for the trivial state (left figures) and for
a wave with approximate amplitude $2.3\times10^{-2}$ (right figures). The top two figures show the
situation at the zero/pole collision
point, $\mu=0.3585$. The (red) circles correspond to polynomial eigenvalues, and the (red) cross
is the spurious zero of the Krein eigenvalues. The (green) vertical lines are poles
of the Krein matrix. In each quadrant the bottom figure is a blow-up
of the top figure near the polynomial eigenvalues of interest. Note the existence of a Hamiltonian-Hopf bifurcation upon the perturbation.}
\label{f:KreinEvalCollide2}
\end{center}
\end{figure}

A picture of the dispersion curves for the full problem,
\[
z_n(\mu)=-(n+\mu)d(n,\mu),\quad n\in\Z,
\]
is provided in \cref{f:FifthOrderKdVSpectra} for relevant values of $n$.
If the curve is dotted, then for fixed $\mu$ that corresponds to a polynomial
eigenvalue with negative Krein signature. The solid curves correspond to
polynomial eigenvalues with positive Krein signature. There are two possible
values for which a bubble may appear:
\[
z_{-2}(\mu)=z_{+1}(\mu)\quad\leadsto\quad
\mu=\frac1{10}\left(5-\sqrt{5(2\sqrt{129}-21)}\right)\sim0.20711,
\]
and
\[
z_0(\mu)=z_{-2}(\mu)\quad\leadsto\quad
\mu=1-\frac15\sqrt{10}\sim0.36754.
\]
Consequently, for small waves there are at most two instability bubbles.
For a wave with approximate amplitude $2.3\times10^{-2}$ we have the spectral
magnitude plots of \cref{f:specamp}. There we show the maximal value of
the absolute value of the real part of a polynomial eigenvalue for various
values of $\mu$ near the predicted bifurcation points, $\mu\sim0.207$ and
$\mu\sim0.368$. In both cases the range of $\mu$ values for which there is an
instability is $\calO(10^{-3})$.

We conclude by showing plots of the Krein eigenvalues for the situation in the right panel, $\mu\sim0.36$. In \cref{f:KreinEvalCollide1} we see
a plot of the Krein eigenvalues for $\mu\sim0.368$. The panel on the left
shows the plot for the trivial state, and the panel on the right shows the
plot for a small wave. Since this value of $\mu$ is not associated with an
instability (see right panel of \cref{f:specamp}), the zeros of the Krein
eigenvalues are purely real. One of the zeros corresponds to a polynomial eigenvalue with negative Krein index. In
\cref{f:KreinEvalCollide2} we see a plot of the Krein eigenvalues for
$\mu=0.3585$. The panel on the left shows the plot for the trivial state, and
the panel on the right shows the plot for a small wave.
Here there is not a
zero/pole collision for the Krein eigenvalues. On the bottom left figure we see a polynomial eigenvalue with negative Krein signature, and a removable singularity which corresponds to a polynomial eigenvalue with positive Krein signature. For $\epsilon>0$ a zero of the Krein eigenvalue emerges from the pole (e.g., see the bottom right figure in \cref{f:KreinEvalCollide1}), and this zero corresponds to a polynomial eigenvalue with positive Krein signature. As $\epsilon$ increases these two zeros of the Krein eigenvalue collide, and leave the real axis through a saddle-node bifurcation. Since the
zeros of the Krein eigenvalues now have nonzero imaginary part, for this value of $\mu$ there is a spectral instability (see right panel of \cref{f:specamp}).

\section{Application: location of small eigenvalues}\label{s:5}

The goal here is to use the Krein matrix to locate small polynomial eigenvalues.
We start by assuming that the operator $\calA_0$ has a collection of arbitrarily small eigenvalues. These eigenvalues may arise, e.g., when looking at,
\begin{enumerate}
  \item modulational stability problems for spatially periodic waves
  \item sideband (transverse) stability problems for uni-directional waves
  \item interaction stability problems for multi-pulses.
\end{enumerate}
For multi-pulse problems, the stability of multi-pulses that arise from a stable single pulse is determined solely by the location of eigenvalues near the origin \cite{sandstede:som98}. These eigenvalues reflect interaction properties of the individual pulses which comprise a multi-pulse. Multi-pulses have been a topic of interest since at least \cite{Evans1982}, which proves the existence of a double pulse traveling wave in nerve axon equations. A summary of early results related to multi-pulses can be found in \cite[Section 1]{sandstede:som98}.

\begin{assumption}\label{ass:smalleval}
For each $\epsilon>0$ there exist $N$ eigenvalues of $\calA_0=\calA_0(\epsilon)$, say $\mu_1,\dots,\mu_N$, which satisfy $|\mu_j|<\epsilon$. The number $N$ is independent of $\epsilon$. Moreover, there exists a positive constant $C$, independent of $\epsilon$, such that all other eigenvalues of $\calA_0$ satisfy $|\mu|>C$.
\end{assumption}

We will let $s_1,\dots,s_N$ be the normalized set of associated eigenfunctions,
\[
\calA_0s_j=\mu_js_j,\quad\langle s_j,s_k\rangle=\delta_{jk},
\]
and the subspace $S$ used in the construction of the Krein matrix will be spectral subspace, $S=\Span\{s_1,\dots,s_N\}$. Letting $P_S$ represent the spectral projection for $\calA_0$, we have,
\[
P_S\calA_0=\calA_0P_S,\quad
P_{S^\perp}\calA_0=\calA_0P_{S^\perp}.
\]
The Krein matrix, $\vK_S(z)$ for $z=-\rmi\lambda$, associated with this subspace is given in \cref{thm:krein}, and the eigenvalues for the star-even operator are found by solving,
\begin{equation}\label{e:51aa}
\vK_S(z)\vx=0.
\end{equation}
We start with a preliminary result concerning the part of the Krein matrix which generates poles.

\begin{lemma}\label{l:51}
There exists a constant $C_0>0$, independent of $\epsilon$, such that for $n=1,2$ and $|z|<1/C_0,\,P_{S^\perp}P_n(\rmi z)P_{S^\perp}$ is invertible. Moreover, for $|z|$ sufficiently small there is the expansion,
\[
\left(P_{S^\perp}P_n(\rmi z)P_{S^\perp}\right)^{-1}=
\left[\calI+\calO(|z|)\right]\left(P_{S^\perp}\calA_0P_{S^\perp}\right)^{-1}.
\]
\end{lemma}

\begin{proof}
First suppose $n=1$. Then
\[
P_{S^\perp}P_1(\rmi z)P_{S^\perp}=P_{S^\perp}\calA_0P_{S^\perp}
\left[\calI+z
\left(P_{S^\perp}\calA_0P_{S^\perp}\right)^{-1}P_{S^\perp}(\rmi\calA_1)P_{S^\perp}\right].
\]
The operator $P_{S^\perp}\calA_0P_{S^\perp}$ is invertible with bounded inverse, as $S$ is a spectral subspace associated with the small eigenvalues. Since $\left(P_{S^\perp}\calA_0P_{S^\perp}\right)^{-1}P_{S^\perp}(\rmi\calA_1)P_{S^\perp}$ is a compact operator, it too is uniformly bounded. Setting,
\[
C_0=\|\left(P_{S^\perp}\calA_0P_{S^\perp}\right)^{-1}P_{S^\perp}(\rmi\calA_1)P_{S^\perp}\|,
\]
the operator $\calI+z
\left(P_{S^\perp}\calA_0P_{S^\perp}\right)^{-1}P_{S^\perp}(\rmi\calA_1)P_{S^\perp}$ is invertible for $|z|<1/C_0$. Moreover, a first-order Taylor expansion provides,
\[
\left(\calI+z
\left(P_{S^\perp}\calA_0P_{S^\perp}\right)^{-1}P_{S^\perp}(\rmi\calA_1)P_{S^\perp}\right)^{-1}=\calI+\calO(|z|).
\]
Taking the inverse yields the desired result.

If $n=2$ a similar argument gives the same result once one writes,
\[
P_{S^\perp}P_2(\rmi z)P_{S^\perp}=P_{S^\perp}\calA_0P_{S^\perp}
\left[\calI+z
\left(P_{S^\perp}\calA_0P_{S^\perp}\right)^{-1}P_{S^\perp}\left(\rmi\calA_1-z\calA_2\right)P_{S^\perp}\right],
\]
and then notes that by assumption $\left(P_{S^\perp}\calA_0P_{S^\perp}\right)^{-1}P_{S^\perp}\calA_2P_{S^\perp}$ is also compact.
\end{proof}

Since $P_{S^\perp}P_n(\rmi z)P_{S^\perp}$ is invertible for small $z$, we know through the argument in \cref{s:21} that the following holds:

\begin{corollary}\label{cor:51}
$\lambda_0$ is a small polynomial eigenvalue if and only if $\det\vK_S(z_0)=0$ for $z_0=-\rmi\lambda_0$.
\end{corollary}

%\begin{proof}
%Suppose that $\det\vK_S(z_0)\neq0$, so by \cref{e:51aa} $\vx=\vn$. This is equivalent to the associated eigenfunction, $u_0$, also being an eigenfunction for the operator $P_{S^\perp}P_n(\rmi z_0)P_{S^\perp}$. But, by \cref{l:51} $P_{S^\perp}P_n(\rmi z_0)P_{S^\perp}$ is invertible for small $z_0$. This contradicts the fact that $u_0$ is nontrivial.
%\end{proof}

We now use the result of \cref{l:51} to find an approximation of the Krein matrix for small $z$.

\begin{lemma}\label{l:52}
Suppose that $n=1$. The Krein matrix is analytic for $|z|<1/C_0$. Moreover, if $|z|$ is sufficiently small the Krein matrix has the expansion,
\begin{equation}\label{KSz}
\begin{aligned}
\vK_S(z)&=-z\Big[\diag(\mu_1,\dots,\mu_N)+\overline{z}\left(\rmi\calA_1|_S\right)\\
&\qquad-
\overline{z}^2\left\{-
\calA_1P_{S^\perp}\left(P_{S^\perp}\calA_0P_{S^\perp}\right)^{-1}P_{S^\perp}\calA_1|_S\right\}
+\calO(|z|^3)\Big].
\end{aligned}
\end{equation}
\end{lemma}

\begin{proof}
Analyticity follows from the fact that $P_{S^\perp}P_1(\rmi z)P_{S^\perp}$ is invertible for $|z|<1/C_0$.
Regarding the expansion, we first note that for the first term in the Krein matrix,
\[
\left(\calP_1(\rmi z)|_S\right)_{jk}=\langle s_j,[\calA_0+z(\rmi\calA_1)]s_k\rangle=
\mu_k\langle s_j,s_k\rangle+\overline{z}\langle s_j,(\rmi\calA_1)s_k\rangle,
\]
so upon using the fact the eigenfunctions for $\calA_0$ form an orthonormal basis,
\[
\calP_1(\rmi z)|_S=\diag(\mu_1,\dots,\mu_N)+\overline{z}\left(\rmi\calA_1|_S\right).
\]

Regarding the second term of the Krein matrix, first recall that we saw in the proof of \cref{l:51} that for small $|z|$,
\[
\begin{aligned}
P_{S^\perp}P_1(\rmi z)P_{S^\perp}&=P_{S^\perp}\calA_0P_{S^\perp}
\left[\calI+z
\left(P_{S^\perp}\calA_0P_{S^\perp}\right)^{-1}P_{S^\perp}(\rmi\calA_1)P_{S^\perp}\right]\\
&=P_{S^\perp}\calA_0P_{S^\perp}\left[\calI+\calO(|z|)\right],
\end{aligned}
\]
so upon using a Taylor expansion in $z$,
\[
\left(P_{S^\perp}P_1(\rmi z)P_{S^\perp}\right)^{-1}=
\left[\calI+\calO(|z|)\right]\left(P_{S^\perp}\calA_0P_{S^\perp}\right)^{-1}.
\]
Second, since $P_{S^\perp}$ is a spectral projection, for any $s\in S$,
\[
P_{S^\perp}\calP_1(\rmi z)s=zP_{S^\perp}(\rmi\calA_1)s.
\]
Combining these two facts,
\[
\begin{aligned}
&\left(\calP_1(\rmi z)P_{S^\perp}(P_{S^\perp}\calP_1(\rmi z)P_{S^\perp})^{-1}P_{S^\perp}\calP_1(\rmi z)|_{S}\right)_{jk} \\
&\qquad\qquad=
\langle s_j,\calP_1(\rmi z)P_{S^\perp}(P_{S^\perp}\calP_1(\rmi z)P_{S^\perp})^{-1}P_{S^\perp}\calP_1(\rmi z)s_k\rangle \\
     &\qquad\qquad=\langle P_{S^\perp}\calP_1(-\rmi\overline{z})^\rma s_j,(P_{S^\perp}\calP_1(\rmi z)P_{S^\perp})^{-1}P_{S^\perp}\calP_1(\rmi z)s_k\rangle\\
     &\qquad\qquad=\langle\overline{z}P_{S^\perp}(\rmi\calA_1)s_j,
z\left[\calI+\calO(|z|)\right]\left(P_{S^\perp}\calA_0P_{S^\perp}\right)^{-1}
P_{S^\perp}(\rmi\calA_1)s_k\rangle\\
&\qquad\qquad=\overline{z}^2\langle s_j,(\rmi\calA_1)P_{S^\perp}\left(P_{S^\perp}\calA_0P_{S^\perp}\right)^{-1}
P_{S^\perp}(\rmi\calA_1)s_k\rangle+\calO(|z|^3),
\end{aligned}
\]
which provides,
\[
\begin{aligned}
&\calP_1(\rmi z)P_{S^\perp}(P_{S^\perp}\calP_1(\rmi z)P_{S^\perp})^{-1}P_{S^\perp}\calP_1(\rmi z)|_{S}=\\
&\qquad\qquad\overline{z}^2
(\rmi\calA_1)P_{S^\perp}\left(P_{S^\perp}\calA_0P_{S^\perp}\right)^{-1}P_{S^\perp}(\rmi\calA_1)|_S
+\calO(|z|^3).
\end{aligned}
\]

The final result follows upon combining the above two calculations.
\end{proof}

Upon setting $\gamma=\rmi\overline{z}$ the bracketed part of the Krein matrix \cref{KSz} is approximated by a quadratic star-even polynomial matrix,
\[
\diag(\mu_1,\dots,\mu_N)+\gamma\left(\calA_1|_S\right)+
\gamma^2
\left[-\calA_1P_{S^\perp}\left(P_{S^\perp}\calA_0P_{S^\perp}\right)^{-1}P_{S^\perp}\calA_1|_S\right].
\]
Since $|\mu_j|=\calO(\epsilon)$, the polynomial eigenvalues for this matrix will be $\calO(\epsilon^{1/2})$; consequently, the smallness assumption of \cref{l:51} regarding the polynomial eigenvalues is satisfied. Moreover, to leading order the polynomial eigenvalues are found by ignoring the middle term, so the small polynomial eigenvalues are found by solving the generalized linear eigenvalue problem,
\begin{equation}\label{e:52aa}
\diag(\mu_1,\dots,\mu_N)\vv=\alpha
\left[-\calA_1P_{S^\perp}\left(P_{S^\perp}\calA_0P_{S^\perp}\right)^{-1}P_{S^\perp}\calA_1|_S\right]\vv,\quad
\alpha=-\gamma^2=\overline{z}^2.
\end{equation}
In conclusion, the $N$ small eigenvalues for $\calA_0$ will generate $2N$ small polynomial eigenvalues, and to leading order these small polynomial eigenvalues are realized as the eigenvalues for the generalized eigenvalue problem \cref{e:52aa}. Since $\det\vK_S(\gamma)$ is analytic, and the winding number is invariant under small perturbations, the result is robust; in other words, we can conclude that there will be precisely $2N$ small polynomial eigenvalues for $\calP_1(\rmi z)$, and these polynomial eigenvalues will be $\calO(\epsilon^{1/2})$.

\begin{remark}
If $S=\ker(\calA_0)$, then under the assumption $\calA_1|_{\ker(\calA_0)}$ is the zero matrix,
\[
-\calA_1P_{S^\perp}\left(P_{S^\perp}\calA_0P_{S^\perp}\right)^{-1}P_{S^\perp}\calA_1|_S=
-\calA_1\calA_0^{-1}\calA_1|_{\ker(\calA_0)},
\]
which is precisely the constraint matrix associated with the Hamiltonian-Krein index calculation for linear star-even problems, see equation \cref{e:11}.
\end{remark}

If $n=2$, then an argument similar to that provided for \cref{l:52} provides the approximate Krein matrix for small $|z|$.  The details of the proof will be left for the interested reader.

\begin{lemma}\label{l:53}
Suppose that $n=2$. If $|z|$ is sufficiently small the Krein matrix can be written,
\[
\begin{aligned}
&\vK_S(z)=-z\Big[\diag(\mu_1,\dots,\mu_N)+\overline{z}\left(\rmi\calA_1|_S\right)\\
&\qquad\qquad
-\overline{z}^2\left(\calA_2-
\calA_1P_{S^\perp}\left(P_{S^\perp}\calA_0P_{S^\perp}\right)^{-1}P_{S^\perp}\calA_1\right)|_S
+\calO(|z|^3)\Big].
\end{aligned}
\]
\end{lemma}

\begin{remark}
If $S=\ker(\calA_0)$, then under the assumption $\calA_1|_{\ker(\calA_0)}$ is the zero matrix,
\[
\left(\calA_2-
\calA_1P_{S^\perp}\left(P_{S^\perp}\calA_0P_{S^\perp}\right)^{-1}P_{S^\perp}\calA_1\right)|_S=
\left(\calA_2-
\calA_1\calA_0^{-1}\calA_1\right)|_{\ker(\calA_0)},
\]
which is precisely the constraint matrix associated with the Hamiltonian-Krein index calculation for quadratic star-even problems, see equation \cref{e:12}.
\end{remark}

\section{Example: suspension bridge equation}\label{s:6}

Motivated by observations of traveling waves on suspension bridges, McKenna and Walter \cite{McKenna1990} proposed the model,
\begin{equation}\label{susp}
\partial_t^2u + \partial_x^4u + u^+ - 1 = 0,
\end{equation}
to describe waves propagating on an infinitely long suspended beam, where $u^+ = \max(u, 0)$. To reduce the complexity due to the nonsmooth term $u^+$, Chen and McKenna \cite{Chen1997} introduced the regularized equation,
\begin{equation}\label{susp2}
\partial_t^2u + \partial_x^4u + \rme^{u-1} - 1 = 0.
\end{equation}
Making the change of variables $u - 1 \mapsto u$ in \cref{susp2}, so that localized solutions will decay to a baseline of 0, we will consider the equation,
\begin{equation}\label{susp3}
\partial_t^2u + \partial_x^4u +  \rme^{u} - 1 = 0.
\end{equation}
Writing this in a co-moving frame with speed $c$ by letting $\xi = x - ct$, equation \cref{susp3} becomes
\begin{equation}\label{suspc}
\partial_t^2u - 2 c\partial_{xt}^2u +\partial_x^4u +  c^2\partial_x^2u + \rme^{u} - 1 = 0,
\end{equation}
where we have renamed the independent variable back to $x$.

An equilibrium solution to \cref{suspc} satisfies the ODE,
\begin{equation}\label{eqODE}
\partial_x^4u +  c^2\partial_x^2u + \rme^{u} - 1 = 0.
\end{equation}
Smets and van den Berg \cite[Theorem 11]{Smets2002} prove the existence of a localized, symmetric solution $U(x)$ to \cref{eqODE} for almost all wavespeeds $c \in (0, \sqrt{2})$. Van den Berg et al. \cite[Theorem~1]{Berg2018} use a computer-assisted proof technique to show existence of such solutions to \cref{eqODE} for all speeds $c$ with $c^2 \in [0.5, 1.9]$.
Equation \cref{eqODE} can be written as a first-order system in the standard way as
\begin{equation}\label{suspsystem}
Y' = F(Y; c),
\end{equation}
where $Y = (y_1, y_2, y_3, y_4) = (u, \partial_x u, \partial_x^2 u, \partial_x^3 u)$ and $F: \R^4 \times \R \rightarrow \R^4$, given by
\begin{equation}\label{suspF}
F(y_1, y_2, y_3, y_4; c) = (y_2, y_3, y_4, -c^2 y_3 - \rme^{y_1} + 1),
\end{equation}
is smooth. Furthermore, $F$ has the reversible symmetry $F(R(Y)) = -R(F(Y))$, where $R: \R^4 \rightarrow \R^4$ is the standard reversor operator defined by
\[
R(y_1, y_2, y_3, y_4) = (y_1, -y_2, y_3, -y_4).
\]

Equation \cref{suspsystem} is Hamiltonian with energy $H:\R^4 \times \R \rightarrow \R$ given by
\begin{equation}\label{suspH}
H(Y; c) = y_2 y_4 - \frac{1}{2}y_3^2 + \frac{c^2}{2}y_2^2 + \rme^{y_1} - y_1.
\end{equation}
We note that for all $c \in (0, \sqrt{2})$, $Y = 0$ is a hyperbolic equilibrium of \cref{suspsystem}, and the spectrum of $DF(0; c)$ is the quartet of eigenvalues
\begin{align}\label{specA00}
\mu = \pm \sqrt{\frac{-c^2 \pm \sqrt{c^4 - 4}}{2} } = \pm \alpha \pm \rmi\beta,
\end{align}
for $\alpha, \beta > 0$. Thus the equilibrium at 0 has a two-dimensional stable manifold $W^s(0; c)$ and two-dimensional unstable manifold $W^u(0; c)$.

We take the following hypothesis concerning the existence of a localized, symmetric, primary pulse solution to \cref{suspsystem}.
\begin{hypothesis}\label{Uexistshyp}
For some $c_0 \in (0, \sqrt{2})$, there exists a nontrivial, symmetric homoclinic orbit solution $Y(x; c_0) \in W^s(0; c_0) \cap W^u(0; c_0) \subset H^{-1}(0; c_0)$ to \cref{suspsystem}. Furthermore, the stable manifold $W^s(0; c_0)$ and the unstable manifold $W^u(0; c_0)$ intersect transversely in $H^{-1}(0; c_0)$ at $Y(0; c_0)$.
\end{hypothesis}

We have the following result, which proves the existence of homoclinic orbits $Y(x; c)$ for $c$ near $c_0$.
\begin{lemma}\label{lemma:cinterval}
Assume \cref{Uexistshyp}. Then there exists an open interval $(c_-, c_+)$ containing $c_0$ such that for all $c \in (c_-, c_+)$ the stable and unstable manifolds $W^s(0; c)$ and $W^u(0; c)$ have a one-dimensional transverse intersection in $H^{-1}(0; c)$ which is a homoclinic orbit $Y(x; c)$. $Y(x; c)$ is symmetric with respect to the standard reversor operator $R$, and the map $c \mapsto Y(x; c)$ from $(c_-, c_+)$ to $C(\R, \R^4)$ is smooth.
\end{lemma}
\begin{proof}
Briefly, $Y(0; c_0) \neq 0$, and it follows from the form of the Hamiltonian in \cref{suspH} that $\nabla_Y H(Y(0; c_0); c_0) \neq 0$. By the implicit function theorem, for $c$ close to $c_0$, the 0-level set $H^{-1}(0; c)$ contains a smooth 3-dimensional manifold $K(c)$, with $K(c_0)$ containing $Y(0; c_0)$. The result follows from the transverse intersection of $W^s(0; c_0)$ and $W^u(0; c_0)$ in $K(c_0) \subset H^{-1}(0; c_0)$, the smoothness of $F$, and the implicit function theorem. Symmetry with respect to the reversor $R$ follows from symmetry of $Y(0; c_0)$ and the reversibility of \cref{suspsystem}.
\end{proof}

\begin{remark}We can choose $(c_-, c_+)$ to be the maximal open interval for which \cref{lemma:cinterval} holds. Given the existence results of \cite{Smets2002,Berg2018} and our own numerical analysis, it is likely that $(c_-, c_+) = (0, \sqrt{2})$.
\end{remark}

It follows from the stable manifold theorem that for $c \in (c_-, c_+)$, $Y(x; c)$ is exponentially localized, i.e. for any $\epsilon > 0$,
\begin{align}\label{Yexploc}
|Y(x; c)| &\leq C e^{-(\alpha - \epsilon)|x|} && x \in \R,
\end{align}
where $\alpha$ depends on $c$ and is given by \cref{specA00}. In the next lemma, we prove that $\partial_c Y(x; c)$ is also exponentially localized.

\begin{lemma}\label{lemma:Ycexploc}
The function $\partial_c Y(x; c)$ is exponentially localized, i.e. for each $c \in (c_-, c_+)$ and $\epsilon > 0$ there is a constant $C$ so that
\begin{align}\label{Ycexploc}
|\partial_c Y(x; c)| &\leq C e^{-(\alpha - \epsilon)|x|} && x \in \R.
\end{align}

\begin{proof}
Fix $c \in (c_-, c_+)$. Since $Y(x; c)$ solves equation \cref{suspsystem}, $Y(x; c) \in C^1(\R, \R^4)$. Differentiating \cref{suspsystem} with respect to $c$, which we can do by \cref{lemma:cinterval}, we have
\begin{equation}\label{Ycprime}
Y_c'(x; c) = F_Y(Y(x;c); c) Y_c(x; c) + F_c(Y(x;c); c).
\end{equation}
It follows from the form of $F$ given in \cref{suspF} and \cref{Yexploc} that $F_c(Y(x;c); c)$ is exponentially localized, i.e. for each $\epsilon > 0$ there is a constant $C$ with
\begin{align}\label{Fcexploc}
|F_c(Y(x;c); c)| &\leq C e^{-(\alpha - \epsilon)|x|} && x \in \R.
\end{align}
Define the linear operator $\calL$ by
\begin{equation}\label{suspdefL}
\calL: C^1(\R, \R^4) \to C^0(\R, \R^4),\quad
Z \mapsto \calL Z = \frac{dZ}{dx} - F_Y(Y(x;c); c) Z.
\end{equation}
By equation \cref{Ycprime}, $F_c(Y(x;c); c) \in \Ran \calL$. Since $DF(0; c)$ is hyperbolic, \cite[Lemma~4.2]{Palmer1984} and the roughness theorem for exponential dichotomies \cite{Coppel1978} imply that $\calL$ is Fredholm with index 0. By \cref{Uexistshyp}, we have $\ker \calL = \Span\{Y'(x; c)\}$. Thus the set of all bounded solutions to \cref{Ycprime} is $\{Y_c(x; c) + \R Y'(x; c)\}$.

Next, we recast the problem in an exponentially weighted space. Choose any $\epsilon \in (0,\alpha)$ and let $\eta(x)$ be a standard mollifier function \cite[Section~C.5]{Evans2010}, then we consider
\begin{equation}\label{defYZ}
Y(x; c) = Z(x; c) e^{-(\alpha - \epsilon)r(x)}
\end{equation}
with $r(x) = \eta(x) * |x|$. Note that $r(x)$ is smooth and that $r(x) = |x|$ and $r'(x) = 1$ for $|x| > 1$. Substituting \cref{defYZ} into \cref{Ycprime} and simplifying, we obtain the weighted equation
\begin{equation}\label{Zcprime}
Z'(x; c) = [F_Y(Y(x;c); c) + (\alpha - \epsilon) r'(x) ] Z(x; c) + e^{(\alpha - \epsilon)r(x)} F_c(Y(x;c); c).
\end{equation}
By \cref{Fcexploc} and the definition of $r(x)$, the function $e^{(\alpha - \epsilon)r(x)} F_c(Y(x;c); c)$ is bounded. Define the weighted linear operator $\calL_{\alpha - \epsilon}: C^1(\R, \R^4) \mapsto C^0(\R, \R^4)$ by
\begin{equation}\label{suspdefLalpha}
\calL_{\alpha - \epsilon} = \frac{d}{dx} - F_Y(Y(x;c); c) - (\alpha - \epsilon) r'(x) \calI.
\end{equation}
Equations \cref{Zcprime} and \cref{Fcexploc} imply that $e^{(\alpha - \epsilon)r(x)} F_c(Y(x;c); c) \in \Ran \calL_{\alpha - \epsilon}$. Since $DF(0; c)-(\alpha-\epsilon)\calI$ is still hyperbolic with the same unstable dimension as $DF(0; c)$, it follows again from \cite[Lemma 4.2]{Palmer1984} that $\calL_{\alpha - \epsilon}$ is Fredholm with index 0. Next, we note that the stable-manifold theorem implies that $Y'(x; c)$ is exponentially localized so that
\begin{align}\label{Yprimeloc}
|Y'(x; c)| &\leq C e^{-(\alpha - \epsilon)|x|} && x \in \R.
\end{align}
Since $Y'(x; c) \in \ker \calL$ and $e^{(\alpha - \epsilon)r(x)} Y'(x; c)$ is bounded, it is straightforward to verify that $e^{(\alpha - \epsilon)r(x)} Y'(x; c) \in \ker \calL_{\alpha - \epsilon}$. Since any element in $\ker \calL_{\alpha - \epsilon}$ gives an element of $\ker \calL$ via (\ref{defYZ}), we conclude that
\[
\ker \calL_{\alpha - \epsilon} = \Span\{ e^{(\alpha - \epsilon)r(x)} Y'(x; c)\}.
\]
Since $e^{(\alpha - \epsilon)r(x)} F_c(Y(x;c); c) \in \Ran \calL_{\alpha - \epsilon}$, the set of all bounded solutions to \cref{Zcprime} is $\{Z_c(x; c) + \R e^{(\alpha - \epsilon)r(x)} Y'(x;c)\}$, which implies that $Y_c(x; c) = Z_c(x; c) e^{-(\alpha - \epsilon)r(x)}$ is exponentially localized as claimed.
\end{proof}
\end{lemma}

For $c \in (c_-, c_+)$, let
\begin{equation}\label{suspU}
U(x; c) = y_1(x; c).
\end{equation}
Then $U(x; c)$ is an even function and is an exponentially localized traveling wave solution solution to \cref{suspc}. For the remainder of this section, we will fix $c \in (c_-, c_+)$ and write the primary pulse solution corresponding to wavespeed $c$ as $U(x)$. We are interested in the existence and stability of multi-pulse equilibrium solutions to \cref{suspc}. A multi-pulse is a localized, multi-modal solution $U_n(x)$ to \cref{eqODE} which resembles multiple, well-separated copies of the primary pulse $U(x)$.

\subsection{Existence of pulses}

First, we look at the existence of such pulses. The linearization of \cref{eqODE} about a given solution $U_*$ of \cref{eqODE} is the operator $\calA_0(U^*): H^4(\R) \subset L^2(\R) \mapsto L^2(\R)$, given by
\begin{equation}\label{defA0}
\calA_0(U^*) = \partial_x^4 + c^2 \partial_x^2 + \rme^{U_*}.
\end{equation}
It follows from \cref{lemma:cinterval} that $\calA_0(U)$ has a one-dimensional kernel spanned by $\partial_x U(x)$. Since $\calA_0(U)$ is self-adjoint, its spectrum is real. We take the following additional hypothesis concerning the point spectrum of $\calA_0(U)$.

\begin{hypothesis}\label{A0hyp}
The following hold concerning the spectrum of $\calA_0(U)$.
\begin{enumerate}
\item $\rmn[\calA_0(U)]=1$, i.e. $\calA_0(U)$ has a unique, simple negative eigenvalue $\lambda_-$.
\item There exists $\delta_0 > 0$ such that the only spectrum of $\calA_0(U)$ in $(-\infty, \delta_0)$ is two simple eigenvalues at $0$ and $\lambda_-$.
\end{enumerate}
\end{hypothesis}

We now have the following theorem, which is adapted from \cite[Theorem~3.6]{sandstede:iol97}.
In all that follows, the norm $||\cdot||_\infty$ is the supremum norm on $C(\R)$, $\langle \cdot, \cdot \rangle$ is the inner product on $L^2(\R)$, and $|| \cdot ||$ is the norm on $L^2(\R)$ induced from the inner product.

\begin{theorem}\label{multiexist}
Assume \cref{Uexistshyp} and \cref{A0hyp}, and let $\delta_0 > 0$ be as in \cref{A0hyp}.
Fix a wavespeed $c$, and let $U(x)$ be an exponentially localized solution to \cref{eqODE}. Then for any $n \geq 2$ and any sequence of nonnegative integers $k_1, \dots, k_{n-1}$ with at least one of the $k_j \in \{0, 1 \}$, there exists a nonnegative integer $m_0$ and $\delta > 0$ with $\delta < \delta_0$ such that:
\begin{enumerate}%[(i)]
	\item For any integer $m$ with $m \geq m_0$, there exists a unique $n-$modal solution $U_n(x)$ to \cref{eqODE} which is of the form
	\begin{equation}\label{qn}
	U_n(x) = \sum_{j = 1}^{n} U^j(x) + r(x),
	\end{equation}
	where each $U^j(x)$ is a translate of the primary pulse $U(x)$. The distance between the peaks of $U^j$ and $U^{j+1}$ is $2 X_j$, where
	\begin{equation*}
	X_j \approx \frac{\pi}{\beta}(2 m + k_j) + \tilde{X},
	\end{equation*}
	$\beta$ is defined in \cref{specA00}, and $\tilde{X}$ is a constant. The remainder term $r(x)$ satisfies
	\begin{equation}\label{rbound}
	\|r\|_\infty \leq C \rme^{-\alpha X_{\mathrm{min}}},
	\end{equation}
	where $\alpha$ is defined in \cref{specA00}, and $X_{\mathrm{min}} = \min\{X_1, \dots, X_{n-1}\}$. This bound holds for all derivatives with respect to $x$.

	\item The point spectrum of the linear operator $\calA_0(U_n)$ on $L^2(\R)$ contains $2n$ eigenvalues in the interval $(-\infty, \delta_0)$, which are as follows:
  \begin{enumerate}
    \item There are $n$ real eigenvalues $\nu_1, \dots, \nu_n$ with $|\nu_j| < \delta$, where $\nu_n = 0$ is a simple eigenvalue, and for $j = 1, \dots, n-1$,
    \[
  	\begin{array}{l}
  	\nu_j < 0 \text{ if } k_j \text{ is odd} \\
  	\nu_j > 0 \text{ if } k_j \text{ is even.}
  	\end{array}
    \]
    We will refer to these as the small magnitude eigenvalues of $\calA_0(U_n)$. For $j = 1, \dots, n-1$, $\nu_j = \mathcal{O}(\rme^{-2\alpha X_{\mathrm{min}}})$, and the corresponding eigenfunctions $s_j$ are given by
  	\begin{equation}\label{sj}
  	s_j = \sum_{k = 1}^{n} d_{jk}\partial_x U^k + w_j,
  	\end{equation}
  	where $d_{jk} \in \C$ are constants, and the remainder terms $w_j$ satisfy
  	\begin{equation}\label{sjwbound}
  	\|w_j\|_\infty \leq C\rme^{-2 \alpha X_{\mathrm{min}}}.
  	\end{equation}
    This bound holds for all derivatives with respect to $x$.
    In particular,
    \[
    \| \partial_x w_j\|_\infty \leq C\rme^{-2 \alpha X_{\mathrm{min}}}.
    \]
    The eigenfunction corresponding to $\nu_n$ is $s_n = \partial_x U_n$.

    \item There are $n$ negative eigenvalues which are $\delta-$close to $\lambda_-$.
  \end{enumerate}

  \item The essential spectrum of $\calA_0(U_n)$ is
    \begin{equation}\label{A0ess}
    \sigma_{\text{ess}}(\calA_0(U_n)) = [1 - c^4/4, \infty).
    \end{equation}
    which is positive and bounded away from 0.

\end{enumerate}
\end{theorem}

\begin{proof}
Using \cref{specA00}, the Hamiltonian \cref{suspH}, the fact that the kernel is simple,
and the fact that the Melnikov integral $M = \int_{-\infty}^\infty (\partial_x U)^2\,\rmd x$ is positive, (a) follows from \cite[Theorem~3.6]{sandstede:iol97}, except for the bound on $r(x)$ and its derivatives with respect to $x$, which follows from \cite{Sanstede1993} and \cite{sandstede:som98}. All eigenvalues are real since $\calA_0(U_n)$ is self-adjoint on $L^2(\R)$.  From Hypothesis \ref{Uexistshyp} and Hypothesis \ref{A0hyp}, $\calA_0(U)$ has a simple eigenvalue at 0 and a simple negative eigenvalue at $\lambda_-$. It follows from \cite{alexander:ati90} that $\calA_0(U_n)$ has $n$ eigenvalues near 0 and $n$ negative eigenvalues near $\lambda_-$. This proves the eigenvalue count on $(-\infty, \delta_0)$ and part (b2). Part (b1) follows from \cite{sandstede:som98}. We can verify directly that $\calA_0(U_n)\partial_x U_n = 0$. Part (c) follows from the Weyl Essential Spectrum Theorem \cite[Theorem~2.2.6]{kapitula:sad13} and \cite[Theorem~3.1.11]{kapitula:sad13}, since $\calA_0(U_n)$ is exponentially asymptotic to $\calA_0(0)$.
\end{proof}

\begin{remark}$\calA_0(U_n)$ may in fact have additional eigenvalues $\lambda$ with $\lambda > \delta_0 > 0$, but these do not matter for the analysis. Our numerical analysis suggests that there are in fact no additional eigenvalues.
\end{remark}

\subsection{Stability of pulses}

Now that we know about the existence of single and multiple pulses, we consider their spectral stability. To determine linear PDE stability of the multi-pulse solutions constructed in Theorem \ref{multiexist}, we look at the linearization of the PDE \cref{suspc} about $U_n(x)$, which is the quadratic operator polynomial $\calP_2(\lambda; U_n): H^4(\R, \C) \subset L^2(\R,\C) \rightarrow L^2(\R,\C)$ given by
\begin{equation}\label{quadeig}
\calP_2(\lambda; U_n) = \calI \lambda^2 + \calA_1 \lambda + \calA_0(U_n)
\end{equation}
where $\calA_0(U_n)$ is defined in \cref{defA0}, $\calI$ refers to the identity, and $\calA_1=-2 c \partial_x$.

First, we consider the essential spectrum. Since $U_n$ is exponentially localized, $\calP_2(\lambda; U_n)$ is exponentially asymptotic to the operator
\begin{equation}\label{quadeig0}
\calP_2(\lambda; 0) = \partial_x^4 + c^2 \partial_x^2 - 2 c \lambda \partial_x + (\lambda^2 + 1).
\end{equation}
By \cite[Theorem~3.1.11]{kapitula:sad13}, $\calP_2(\lambda; U_n)$ is a relatively compact perturbation of $\calP_2(\lambda; 0)$, thus by the Weyl essential spectrum theorem \cite[Theorem~2.2.6]{kapitula:sad13}, $\calP_2(\lambda; U_n)$ and $\calP_2(\lambda; 0)$ have the same essential spectrum. To find the essential spectrum of $\calP_2(\lambda; 0)$, consider the related first-order operator $\calT(\lambda): H^1(\R, \C^4) \subset L^2(\R,\C^4) \rightarrow L^2(\R,\C^4)$ given by
\begin{equation}
\calT(\lambda) = \frac{\rmd}{\rmd x} - \begin{pmatrix}
0 & 1 & 0 & 0 \\
0 & 0 & 1 & 0 \\
0 & 0 & 0 & 1 \\
-1 - \lambda^2 & 2 c \lambda & -c^2 & 0
\end{pmatrix},
\end{equation}
which we obtain by writing $\calP_2(\lambda; 0)$ as a first order system. By a straightforward adaptation of \cite[Theorem A.1]{sandstede:rmi08} (the only difference being the presence of the fourth-order differential operator), the operators $\calT(\lambda)$ and $\calP_2(\lambda; 0)$ have the same Fredholm properties, thus the same essential spectrum. By a straightforward calculation,
\begin{equation}\label{quadess}
\sigma_{\mathrm{ess}}(\calP_2(\lambda; U_n)) = \sigma_{\mathrm{ess}}(\calT(\lambda)) = \{\rmi r : |r| \geq \rho \},
\end{equation}
where $\rho > 0$ is the minimum of the function $\lambda(r) = c r + \sqrt{1 + r^4}$. The value of $\rho$ is positive for $c \in (0, \sqrt{2})$, and $\rho\to0$ as $c\to\sqrt{2}$, so the essential spectrum is purely imaginary and bounded away from 0. Spectral stability thus depends entirely on the point spectrum.

\subsubsection{Single pulse}

Before considering the spectral stability of the $n$-pulse, we must show the stability of the primary pulse, $U(x)$. In addition to Hypothesis \ref{Uexistshyp} and Hypothesis \ref{A0hyp}, our assumptions are:

\begin{hypothesis}\label{PDEexisthyp}
Regarding the PDE \cref{suspc} and the base solution $U(x)$,
\begin{enumerate}
  \item for every initial condition $u(x,0)$ and $\partial_tu(x,0)$ there exists a solution $u(x, t)$ to \cref{suspc} on the interval $I = [0, T]$, where
      \[
      T=T\left(\max\{ ||u(x,0)||, ||\partial_tu(x,0)|| \}\right)
      \]
  \item the constrained energy evaluated on the wave, $d(c)$ (see \cite[Equation~(2.16)]{grillakis:sto87} for the exact expression), is concave up,
\begin{equation}\label{dcc}
d''(c) = -\partial_c\left( c\|\partial_xU\|^2 \right)>0,\quad0<c^2<2.
\end{equation}
\end{enumerate}
\end{hypothesis}
We will provide numerical evidence that these hypotheses are met in \cref{sec:numerics}.

Under these assumptions, we will prove the spectral and orbital stability of the single pulse using the HKI. However, there are first two issues that must be resolved. First, the HKI as discussed in \cref{sec:intro} assumes that $\calA_0$ has a compact resolvent, which is certainly not true for the operator associated with this problem. This compactness assumption is taken primarily for the sake of convenience, and to remove the possibility of point spectrum being embedded in the essential spectrum. However, as seen in the original formulation of the HKI for solitary waves, see \cite{kapitula:cev04,kapitula:ace05}, this is not a necessary condition. It is sufficient to assume that the origin is an isolated eigenvalue, and $\calA_0$ is a higher-order differential operator than $\calA_1$ with $\rmn[\calA_0]<+\infty$.  The interested reader should consult \cite{kapitula:ahk14} for the case where the origin is not isolated.
The second difficulty is that these previous results for solitary waves do not immediately apply to quadratic eigenvalue problems.  However, as seen in \cite[Section~4.1]{bronski:aii14} one can easily convert a quadratic star-even eigenvalue problem into a linear star-even eigenvalue problem, and then apply the index theory to the reformulated problem. Thus, we can conclude the index theory is applicable to the problem at hand, which allows for the following stability result.

\begin{lemma}\label{qstable}
Let $c^2 \in (0, 2)$, and let $U(x)$ be the primary pulse solution to \cref{eqODE}. Then $U(x)$ is spectrally and orbitally stable if and only if
\begin{equation}\label{dcc1}
d''(c) = -\partial_c\left( c\|\partial_x U\|^2 \right)>0,
\end{equation}
where $d(c)$ is defined in \cite[equation (2.16)]{grillakis:sto87}.
\end{lemma}

\begin{proof}
First, equation \cref{dcc1} is well-defined since both $U$ and $\partial_x U$ are smooth in $c$ by \cref{lemma:cinterval}.
Next, we check that the origin is an isolated eigenvalue. The essential spectrum of $\calA_0(U)$ is the same as that of $\calA_0(U_n)$, and is given by \cref{A0ess}, which is positive and bounded away from 0. By assumption, $\calA_0(U)$ has a single negative eigenvalue.

We now use the HKI to complete the proof; in particular, the formulation as presented in equation \cref{e:12}. First, we note that,
\[
\left.\calA_1\right|_{\Span\{\partial_xU\}}=
\langle-2c\partial_x\left(\partial_xU\right),\partial_xU\rangle=0,
\]
where the equality follows from the fact that the primary pulse is even. Since $\calA_2=\calI$ is positive definite, we can write,
\[
\begin{aligned}
K_{\Ham}&=\rmn(\calA_0)-
\rmn\left(\left.\left[\calI-\calA_1\calA_0^{-1}\calA_1\right]\right|_{\Span\{\partial_xU\}}\right)\\
&=1-
\rmn\left(\left.\left[\calI-\calA_1\calA_0^{-1}\calA_1\right]\right|_{\Span\{\partial_xU\}}\right),
\end{aligned}
\]
for by assumption, $\rmn(\calA_0)=1$.

Regarding the second term,
\[
\begin{aligned}
\left.\left[\calI-\calA_1\calA_0^{-1}\calA_1\right]\right|_{\Span\{\partial_xU\}}&=
\|\partial_xU\|^2-
\langle(-2c\partial_x)\calA_0^{-1}(-2c\partial_x)\partial_xU,\partial_xU\rangle\\
&=\|\partial_xU\|^2+
2c\langle\partial_x\calA_0^{-1}(-2c\partial_x^2U),\partial_xU\rangle.
\end{aligned}
\]
Going back to the existence equation \cref{eqODE} and differentiating with respect to $c$ yields,
\[
\calA_0(U)\partial_cU+2c\partial_x^2U=0\,\,\leadsto\,\,
\calA_0(U)^{-1}(-2c\partial_x^2U)=\partial_cU.
\]
Substitution and changing the order of differentiation provides,
\[
\langle\partial_x\calA_0(U)^{-1}(-2c\partial_x^2U),\partial_xU\rangle=
\langle\partial_c\partial_xU,\partial_xU\rangle=\frac{1}{2}\partial_c\|\partial_xU\|^2.
\]
In conclusion,
\[
\left.\left[\calI-\calA_1\calA_0^{-1}\calA_1\right]\right|_{\Span\{\partial_xU\}}=
\|\partial_xU\|^2+c\,\partial_c\|\partial_xU\|^2=
\partial_c\left( c\|\partial_xU\|^2 \right).
\]

We now have for the primary pulse,
\[
K_{\Ham}=1-\rmn\left[\partial_c\left( c\|\partial_xU\|^2 \right)\right].
\]
If $d''(c)<0$, then $K_{\Ham}=1$, and there is one positive real polynomial eigenvalue. If $d''(c)>0$, the HKI is zero. Consequently, the wave is spectrally stable. Appealing to \cite[Theorem~4.1]{bronski:aii14} we can further state that the wave is orbitally stable.
\end{proof}

\subsubsection{$n$-pulse}

We now locate all potentially unstable eigenvalues of \cref{quadeig} for an $n$-pulse. These include polynomial eigenvalues with positive real part, as well as purely imaginary polynomial eigenvalues with negative Krein signature. To accomplish this task we use the HKI in combination with the Krein matrix. First, we compute the HKI for \cref{quadeig}, so that  we have an exact count of the number of potentially unstable polynomial eigenvalues. We then use the Krein matrix to find $(n-1)$ pairs of eigenvalues close to 0; each pair is either real or purely imaginary with negative Krein signature. We refer to these as small magnitude polynomial eigenvalues, or interaction polynomial eigenvalues, since heuristically they result from interactions between neighboring pulses. We then show that the number of potentially unstable interaction polynomial eigenvalues is exactly the same as the HKI, from which we conclude that we have found all of the potentially unstable eigenvalues. By Hamiltonian reflection symmetry, all other point spectrum must be purely imaginary with positive Krein signature.

We start with the calculation of the HKI. By \cref{multiexist} we know that $\calA_0(U_n)$ has precisely $n$ eigenvalues near the origin. Let $0\le n_\rms\le n-1$ represent the number of these eigenvalues which are negative. We have the following result concerning the HKI for the $n$-pulse:

\begin{lemma}\label{lem:HKImulti}
Assume Hypotheses \ref{Uexistshyp}, \ref{A0hyp}, and \ref{PDEexisthyp}, and let $U_n(x)$ be an $n-$modal solution to \cref{eqODE}. Then
\[
K_{\Ham}=n+n_\rms-1.
\]
\end{lemma}

\begin{proof}
From \cref{multiexist} part (b) and the definition of $n_\rms$, $\rmn[\calA_0(U_n)]=n+n_\rms$, so for the HKI,
\[
K_{\Ham}=n+n_\rms-
\rmn\left(\left.\left[\calI-\calA_1\calA_0^{-1}\calA_1\right]\right|_{\Span\{\partial_xU_n\}}\right),
\]
where $\calA_0=\calA_0(U_n)$.
In the proof of \cref{qstable} we saw that when the wave depends smoothly on $c$,
\[
\left.\left[\calI-\calA_1\calA_0^{-1}\calA_1\right]\right|_{\Span\{\partial_xU_n\}}=
\partial_c\left(c\|\partial_x U_n\|^2\right).
\]
Since to leading order the $n$-pulse is $n$ copies of the original pulse, we have
\[
\|\partial_x U_n\|^2=n\|\partial_xU\|^2 + \mathcal{O}(e^{-\alpha X_{\mathrm{min}}}).
\]
Consequently, we can write
\begin{align*}
\partial_c\left(c\|\partial_x U_n\|^2\right) &=
n\partial_c\left(c\|\partial_x U\|^2\right) + \mathcal{O}(e^{-\alpha X_{\mathrm{min}}})\\
&=-nd''(c) + \mathcal{O}(e^{-\alpha X_{\mathrm{min}}}).
\end{align*}
Since $d''(c)>0$ by assumption, we have to leading order,
\[
\partial_c\left(c\|\partial_x U_n\|^2\right)<0.
\]
For sufficiently well-separated pulses the sign will not change even when incorporating the higher-order terms in the asymptotic expansion. The result now follows.
\end{proof}

We now locate the potentially unstable polynomial eigenvalues of the quadratic eigenvalue problem \cref{quadeig}. This will be accomplished through the Krein matrix. For the sake of exposition only we will henceforth assume that each of the small magnitude eigenvalues $\nu_1, \dots, \nu_n$ of $\calA_0(U_n)$ is simple. For each of these eigenvalues, denote the associated normalized eigenfunctions as $s_1, \dots, s_n$. Since $\calA_0(U_n)$ is self-adjoint, these eigenfunctions are pairwise orthogonal. In the construction of the Krein matrix the relevant subspace for the spectral problem is the span of this set of eigenfunctions associated with the small magnitude eigenvalues of $\calA_0$,
\begin{equation}\label{defS}
S = \Span\{s_1, \dots, s_n \}.
\end{equation}

We now present the following theorem, which is the main result of this section.

\begin{theorem}\label{Kreindiag}
Assume Hypotheses \ref{Uexistshyp}, \ref{A0hyp}, and \ref{PDEexisthyp}. Let $U_n(x)$ be an $n-$pulse solution to \cref{eqODE}, and let $\nu_1, \dots, \nu_n$ be the small magnitude eigenvalues of $\calA_0(U_n)$, as defined in \cref{multiexist}. Under a suitable normalization of the eigenfunctions $s_j$, near the origin the Krein matrix has the asymptotic expansion,
\begin{equation}\label{Kreinapprox}
-\frac{\vK_S(z)}{z} = ||\partial_xU||^2 \mathrm{diag} (\nu_1, \dots, \nu_n)
 + d''(c)\vI_n\overline{z}^2 + \mathcal{O}(\rme^{-(3 \alpha/2) X_{\mathrm{min}}}|z| + |z|^3),
\end{equation}
which is diagonal to leading order.
\end{theorem}

The proof of this result is left to \cref{s:kreinproof}. As a corollary, we have the following criteria for spectral stability and instability of the multi-pulse solutions $U_n(x)$.

\begin{corollary}\label{stabcrit}
Let $U_n(x)$ be an $n-$pulse solution to \cref{eqODE} constructed as in \cref{multiexist} using the sequence of nonnegative integers $\{ k_1, \dots, k_{n-1} \}$. Assume the same hypotheses as in \cref{Kreindiag}. Let $\nu_1, \dots, \nu_n$ be the small magnitude eigenvalues of $\calA_0(U_n)$, where $\nu_n = 0$. Then there are $(n-1)$ pairs of eigenvalues of \cref{quadeig} close to 0, which we will term interaction polynomial eigenvalues. These are described as follows. For each $j=1,2,\dots,n-1$,
\begin{enumerate}
  \item if $k_j$ is odd (equivalently, $\nu_j<0$), there is a corresponding pair of purely imaginary interaction polynomial eigenvalues,
  \begin{equation}\label{npulseKreineigs}
	\lambda_j^\pm = \pm \rmi \left( \|\partial_xU\| \sqrt{ \frac{|\nu_j|}{d''(c)} } + \mathcal{O}(\rme^{-(3 \alpha/2) X_{\mathrm{min}}}) \right),
	\end{equation}
  each of which has negative Krein signature
  \item if $k_j$ is even (equivalently, $\nu_j>0$), there is a corresponding pair of real interaction polynomial eigenvalues,
   	\[
	\lambda_j = \pm \left( \|\partial_xU\| \sqrt{ \frac{\nu_j}{d''(c)} } + \mathcal{O}(\rme^{-(3 \alpha/2) X_{\mathrm{min}}}) \right).
	\]
  In particular, there exists a positive, real eigenvalue.
\end{enumerate}
In addition, there is a geometrically simple polynomial eigenvalue at $\lambda=0$ with corresponding eigenfunction $\partial_x U_n$. All other point spectra is purely imaginary, and has positive Krein signature.
\end{corollary}

\begin{remark}
In other words, if all the small magnitude eigenvalues of $\calA_0(U_n)$ are negative, and if the individual pulses are sufficiently well-separated, then the $n$-pulse is spectrally stable; otherwise, it is unstable.
\end{remark}

While we can find the interaction polynomial eigenvalues using Lin's method as in \cite{sandstede:som98}, using the Krein matrix allows us to also determine the Krein signatures of any purely imaginary interaction polynomial eigenvalues. This additional information is needed to ensure that via the HKI all of the potentially unstable point spectrum has small magnitude.

%\begin{remark}The spectral stability result only holds for sufficiently large $X_{\mathrm{min}}$. In particular, if this is not the case, there may be a Krein collision between an interaction eigenvalue and the essential spectrum. This is illustrated numerically in the next section.
%\end{remark}

\begin{proof}
By \cref{cor:51} the small polynomial eigenvalues are found by solving $\det\vK_S(z) = 0$. This is equivalent to finding zeros of the Krein eigenvalues. For $j=1,2,\dots,n$ set,
\[
-\frac{r_j(z)}{z}=||\partial_xU||^2 \nu_j + d''(c) \overline{z}^2+\tilde{r}_j(z),
\]
where
\[
\tilde{r}_j(z) = \mathcal{O}(\rme^{-(3 \alpha/2) X_{\mathrm{min}}}|z| + |z|^3).
\]
Note that the first two terms in $-r_j(z)/z$ are the diagonal entries of the Krein matrix. Since to leading order the Krein matrix is diagonal,  by \cite{Ipsen2008} these are valid asymptotic expressions for the Krein eigenvalues. The small and nonzero polynomial eigenvalues are found by solving,
\begin{equation}\label{eqforz}
||\partial_xU||^2 \nu_j + d''(c) \overline{z}^2+\tilde{r}_j(z)=0,\quad j=1,2,\dots,n.
\end{equation}

First suppose that $z$ is real, so the Krein matrix is Hermitian. The Krein eigenvalues are then real-valued; in particular, the error term, $\tilde{r}_j(z)$, is real-valued. Recall that $d''(c)>0$. Suppose that $\nu_j<0$, and set,
\begin{equation}\label{epsilon2}
\epsilon_j^2 = -\frac{||\partial_xU||^2 \nu_j}{d''(c)} > 0.
\end{equation}
Equation \cref{eqforz} can then be rewritten,
\begin{equation}\label{eqforz2}
z^2 - \epsilon_j^2 + \mathcal{O}(\rme^{-(3 \alpha/2) X_{\mathrm{min}}}|z| + |z|^3) = 0.
\end{equation}
Letting $y = \epsilon_j z$ and noting that $\epsilon_j = \mathcal{O}(\rme^{-\alpha X_{\mathrm{min}}})$, equation \cref{eqforz2} becomes,
\begin{equation}\label{eqforz3}
y^2 - 1 + \mathcal{O}(\epsilon_j^{1/2 }|y| + \epsilon|y^3|) = 0.
\end{equation}
For sufficiently small $\epsilon_j$, equation \cref{eqforz3} has two roots, $y = \pm 1 + \mathcal{O}(\epsilon_j^{1/2})$. Thus, for sufficiently large $X_{\mathrm{min}}$, equation \cref{eqforz} has two solutions,
\[
z_j^\pm = \pm ||\partial_xU|| \sqrt{ -\frac{ \nu_j}{d''(c)} } + \mathcal{O}(\rme^{-(3 \alpha/2) X_{\mathrm{min}}}).
\]
The Krein eigenvalue, $r_j(z)$, has a simple zero at $z_j^\pm$. Since to leading order,
\[
r_j'(z_j^\pm)=-||\partial_xU||^2 \nu_j-3d''(c)(z_j^\pm)^2=2||\partial_xU||^2 \nu_j<0,
\]
each of these polynomial eigenvalues has negative Krein signature.

Now suppose $\nu_j > 0$, and assume $z$ is purely imaginary, $z=\rmi\tilde{z}$. In this case the Krein matrix is no longer Hermitian, which implies that the remainder term associated with each Krein eigenvalue is no longer necessarily real-valued. Define $\epsilon_j^2$ as in \cref{epsilon2}, but this time $\epsilon_j^2 < 0$. The two zeros of the Krein eigenvalue are now,
\[
\tilde{z}_j^\pm = \pm||\partial_xU|| \sqrt{ \frac{ \nu_j}{d''(c)} } + \mathcal{O}(\rme^{-(3 \alpha/2) X_{\mathrm{min}}}),
\]
which to leading order are purely real. Going back to the original problem, there are two interaction polynomial eigenvalues given by,
\[
\lambda_j^\pm=\tilde{z}_j^\pm.
\]
To leading order these eigenvalues are real-valued. Under the assumption that the small magnitude eigenvalues of $\calA_0(U_n)$ are simple, via the asymptotic expansion $\lambda_j^\pm$ will also then be simple. By the Hamiltonian reflection symmetry of the polynomial eigenvalues about the real axis, the fact they are real-valued to leading order implies they are truly real-valued and come in opposite-sign pairs.

Since the kernels of \cref{quadeig} and $\calA_0(U_n)$ are the same, we can verify directly that $\lambda = 0$ is an eigenvalue of \cref{quadeig} with eigenfunction $\partial_x U_n$. We now show that all other point spectra is purely imaginary. We have for the small magnitude polynomial eigenvalues, $k_\rmi^-=2n_\rms$, and $k_\rmr=n-1-n_\rms$. Thus, for the small magnitude polynomial eigenvalues,
\[
k_\rmr+k_\rmi^-=(n-1-n_\rms)+(2n_\rms)=n-1+n_\rms.
\]
By \cref{lem:HKImulti} this is the HKI for the $n$-pulse. Consequently, there are no other point polynomial eigenvalues which have positive real part, or which are purely imaginary and have negative Krein signature.
\end{proof}

\subsection{Numerical results}\label{sec:numerics}

In this section, we show numerical results to illustrate the theoretical results of the previous section. First, we can construct a primary pulse solution $U(x)$ numerically using the string method from \cite{Chamard2011}. The top two panels of \cref{fig:single1} show these solutions for the same values of $c$ as in \cite[Figure 3]{Chen1997}. Next, we compute the spectrum of the operator $\calA_0(U)$ numerically using Matlab's \texttt{eig} function. In the bottom panel of \cref{fig:single1} we note the presence of a simple eigenvalue at the origin and a simple negative eigenvalue, which supports our hypotheses on the spectrum of $\calA_0(U)$. As expected, we also see that the essential spectrum is positive and bounded away from 0.

\begin{figure}[ht]
\centering
\begin{tabular}{cc}
\includegraphics{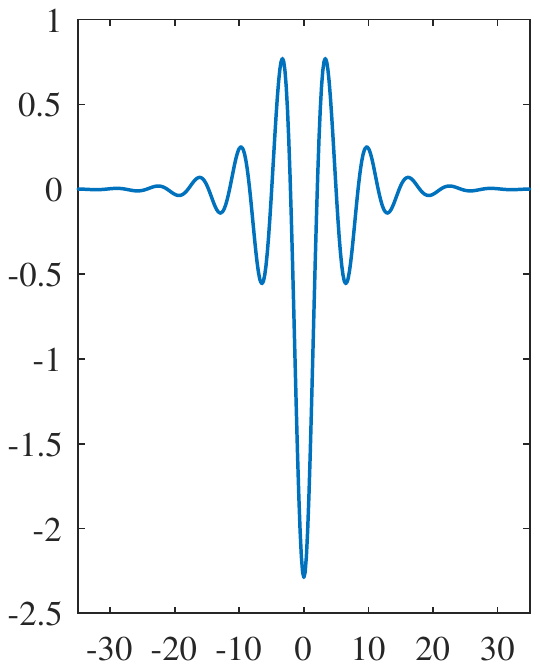}&
\includegraphics{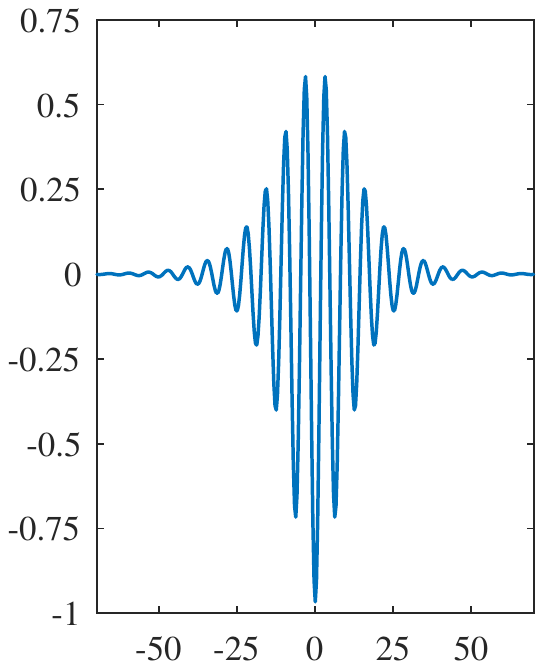}
\end{tabular}
\includegraphics{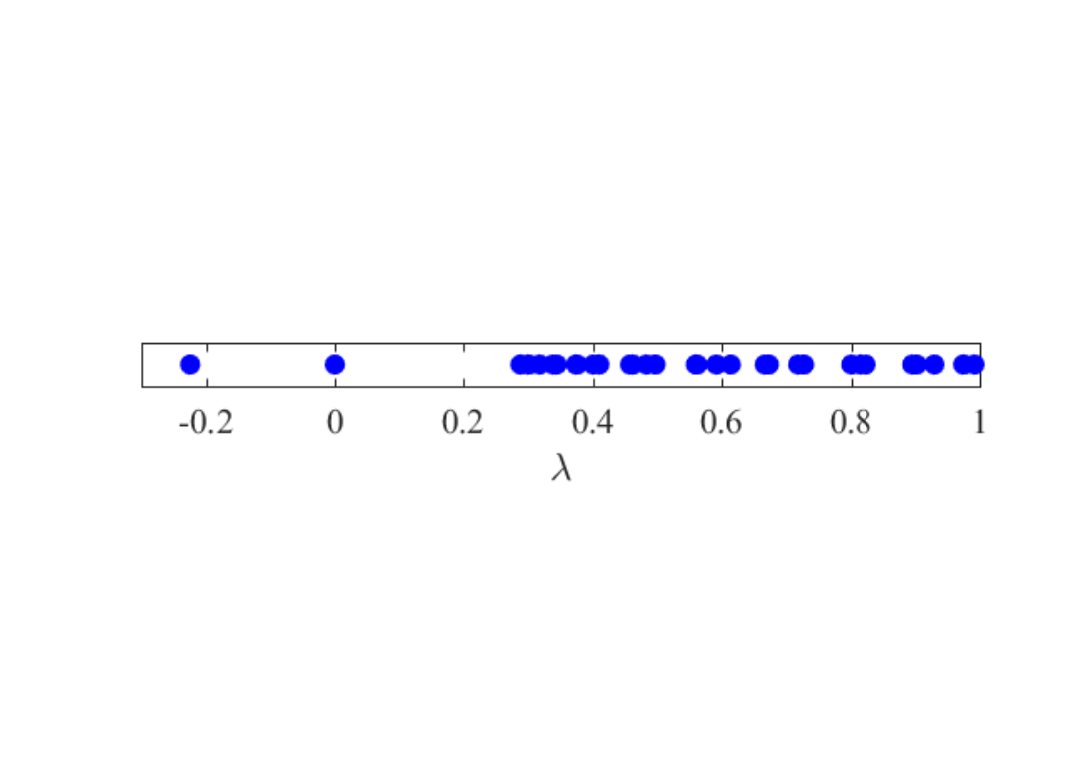}
\caption{Primary pulse solutions $U(x)$ to \cref{eqODE} for $c = 1.354$ (top left) and $c = 1.40$ (top right). In the bottom panel there is the spectrum of $\calA_0(U)$, the linearization of \cref{eqODE} about a single pulse $U(x)$ for $c = 1.3$. For the spectral plot we use finite difference methods with $N = 512$ and periodic boundary conditions. The left boundary of the essential spectrum is $\lambda\sim0.286$. The spectrum to the right of the boundary is discrete instead of continuous because of the boundary conditions. }
\label{fig:single1}
\end{figure}

%\begin{figure}[ht]
%\centering
%\includegraphics{specA0}
%\caption{Spectrum of $\calA_0(U)$, the linearization of \cref{eqODE} about a single pulse $U(x)$ for $c = 1.3$. We use finite difference methods with $N = 512$ and periodic boundary conditions.}
%\label{fig:specA0}
%\end{figure}

We can construct multi-pulse solutions numerically by joining together multiple copies of the primary pulse and using Matlab's \texttt{fsolve} function. Consecutive distances between peaks given by \cref{multiexist}. The first four double pulse solutions are shown in the top two panels of \cref{fig:double}. These double pulses are numbered using the integer $k_1$ from \cref{multiexist}. We verify \cref{multiexist}(b) numerically by computing the spectrum of $\calA_0(U_2)$. The spectrum of $\calA_0(U_2)$ for double pulses 0 and 1 are shown in the bottom two panels of \cref{fig:double}. In both cases, there is an eigenvalue at 0. For double pulse 0, there is an additional positive eigenvalue near 0, and for double pulse 1, there is an additional negative eigenvalue near 0.

\begin{figure}[ht]
\centering
\begin{tabular}{cc}
\includegraphics{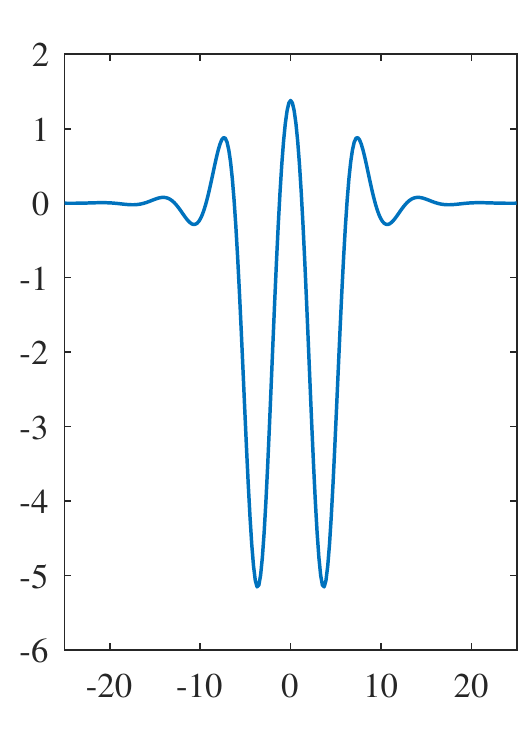}&
\includegraphics{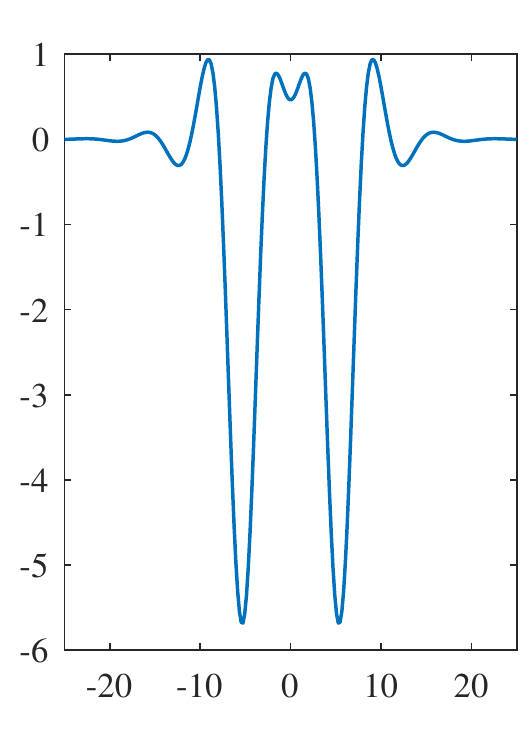}\\
\includegraphics{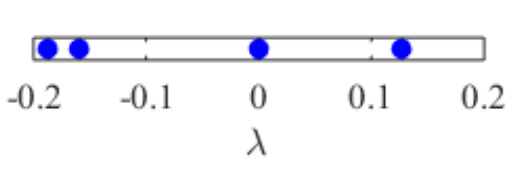}&
\includegraphics{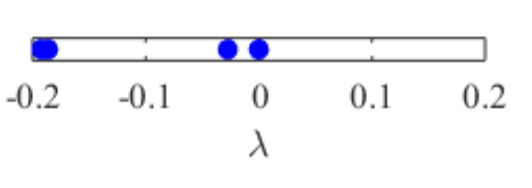}
\end{tabular}
\caption{Double pulse solutions $U_2(x)$ to \cref{eqODE} for $c = 1.2$. The top left panel shows double pulse 0, and the top right panel shows double pulse 1. In the bottom two panels we see the associated spectrum for $\calA_0(U_2)$: double pulse 0 on the left, and double pulse 1 on the right.}
\label{fig:double}
\end{figure}

%\begin{figure}[ht]
%\centering
%\begin{tabular}{cc}
%\includegraphics{specA0d1}&
%\includegraphics{specA0d2}
%\end{tabular}
%\caption{Eigenvalues of $\calA_0(U_2)$ for $c = 1.2$. Double pulses 0 (left) and 1 (right).}
%\label{fig:specA0double}
%\end{figure}

We verify \cref{stabcrit} by computing the polynomial eigenvalues of \cref{quadeig} directly using the Matlab package \texttt{quadeig} from \cite{Hammarling2013}. For double pulse 0, $\calA_0(U_2)$ has one positive small magnitude eigenvalue; thus, by \cref{stabcrit}, equation \cref{quadeig} has a polynomial eigenvalue with positive real part. For double pulse 1, the small magnitude eigenvalue of $\calA_0(U_2)$ is negative; thus by \cref{stabcrit}, since the distance between the two peaks is sufficiently large, the polynomial eigenvalues of \cref{quadeig} are purely imaginary. These are shown in \cref{fig:quadeigdouble}.

\begin{figure}[ht]
\centering
\begin{tabular}{cc}
\includegraphics{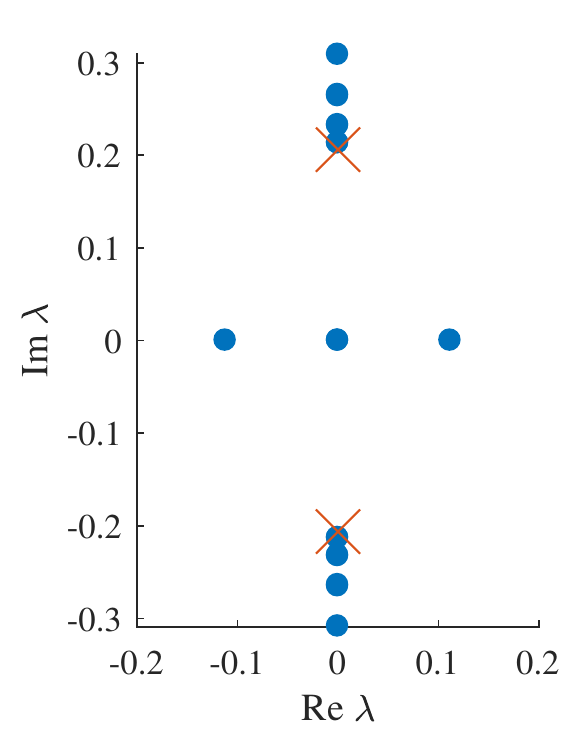}&
\includegraphics{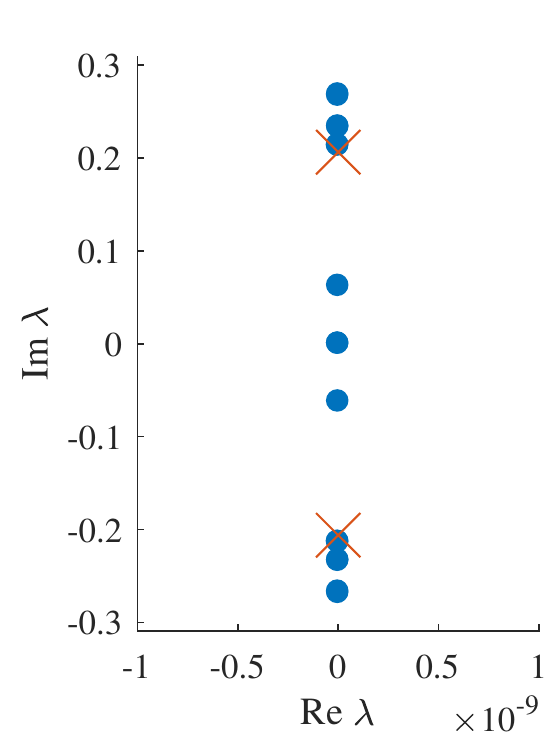}
\end{tabular}
\caption{Polynomial eigenvalues of \cref{quadeig} for double pulses 0 (left) and 1 (right) for $c=1.2$. The eigenvalues are marked with a filled (blue) circle, and the edge of the essential spectrum is marked with a (red) cross. The essential spectrum is discrete instead of continuous because of the boundary conditions. For the right panel the two purely imaginary polynomial eigenvalues nearest the origin have negative Krein signature. Here we use finite difference methods with $N = 512$ and periodic boundary conditions.}
\label{fig:quadeigdouble}
\end{figure}

\subsection{Proof of \cref{Kreindiag}}\label{s:kreinproof}

Using \cref{multiexist}, let $U_n(x)$ be an $n-$modal solution to \cref{eqODE}, and let $\{\nu_1, \dots, \nu_n\}$ be the small magnitude eigenvalues of $\calA_0(U_n)$ with corresponding eigenfunctions $\{ s_1, \dots, s_n \}$. Since $\calA_0(U_n)$ is self-adjoint, the $s_i$ are orthogonal, and for the sake of convenience scale them so that
\begin{equation}\label{orthoeigs}
\langle s_i, s_j \rangle = \|\partial_x U \|^2 \delta_{ij}.
\end{equation}
Typically, we assume these eigenfunctions also have unit length. However, this is not important in the construction of the Krein matrix, nor in the derived properties.
Let $S = \Span\{s_1, \dots s_n\}$.

By \cref{l:53}, and using the normalization of \cref{orthoeigs}, for small $|z|$ the Krein matrix is the $n \times n$ matrix,
\begin{equation}\label{Kreinform}
-\frac{\vK_S(z)}{z} = \|\partial_xU_n\|^2 \text{diag}(\nu_1, \dots, \nu_n) + \overline{z}\vK_1
- \overline{z}^2 ( \|\partial_xU_n\|^2\vI_n - \vK_2) + \mathcal{O}(|z|^3),
\end{equation}
where
\begin{equation}\label{defK1}
(\vK_1)_{jk} = \langle s_j, \rmi\calA_1 s_k \rangle,
\end{equation}
and
\begin{equation}\label{defK2}
(\vK_2)_{jk} = \langle \calA_1 s_j, P_{S^\perp}(P_{S^\perp} \calA_0(U_n)P_{S^\perp})^{-1} P_{S^\perp}\calA_1 s_k \rangle.
\end{equation}
This is, to leading order, a matrix-valued quadratic polynomial in $z$ (and its complex conjugate). %If we take $z \in \R$, the Krein matrix is Hermitian.
The factors $\|\partial_xU_n\|^2$ on the RHS of \cref{Kreinform} come from using the scaling \cref{orthoeigs} for the eigenfunctions $s_i$ of $\calA_0(U_n)$.
We now prove \cref{Kreindiag} in a series of lemmas. In all that follows, $C$ refers to a constant independent of $x$, but it may have a different value each time it is used. The first lemma is a bound on the product of exponentially separated pulses.

\begin{lemma}\label{expseplemma}
Let $U_+(x)$ and $U_-(x)$ be localized pulses which decay exponentially with rate $\alpha$ and whose peaks are separated by a distance $2 X$. We have the following bounds,
\begin{equation}\label{expsepbound1}
\sup_{x \in \R} | U_-(x) U_+(x)|\leq C \rme^{-2 \alpha X},
\end{equation}
and
\begin{equation}\label{expsepbound2}
|\langle U_-(x), U_+(x) \rangle |\leq C \rme^{-(3 \alpha/2) X}.
\end{equation}
\end{lemma}

\begin{proof}
Without loss of generality, let $U_\pm(x)$ be exponentially localized peaks centered at $\pm X$, thus $|U_-(x)| \leq C e^{-\alpha|x + X|}$ and $|U_+(x)| \leq C \rme^{-\alpha|x - X|}$. For $x \in (-\infty, -X]$,
\begin{align}\label{endprodbound}
| U_-(x) U_+(x) | &\leq C \rme^{\alpha(x + X)} \rme^{\alpha(x - X)} = C \rme^{2 \alpha x} \leq C \rme^{-2 \alpha X}
\end{align}
and for $x \in [-X, 0]$,
\begin{align}\label{middleprodbound}
| U_-(x) U_+(x) | &\leq C \rme^{-\alpha(x + X)} \rme^{\alpha(x - X)} = C \rme^{-2 \alpha X}
\end{align}
Bounds on $[0, X]$ and $[X, \infty)$ are similar. Since these are independent of $x$, we obtain the bound \cref{expsepbound1}.

For the bound \cref{expsepbound2}, we split the integral into four pieces.
\begin{equation}
\begin{aligned}
| \langle U_-(x), &U_+(x) \rangle |
\leq \int_{-\infty}^{-X} |U_-(x) U_+(x)| \rmd x + \int_{-X}^0 |U_-(x) U_+(x)|\rmd x \\
&+\int_0^X |U_-(x) U_+(x)| \rmd x +\int_X^\infty |U_-(x) U_+(x)| \rmd x
\end{aligned}
\end{equation}
For the first integral, we use \cref{endprodbound} to get
\begin{align*}
\int_{-\infty}^{-X} |U_-(x) U_+(x)| \rmd x &\leq C \int_{-\infty}^{-X} \rme^{2 \alpha x} \rmd x = C \rme^{-2 \alpha X}
\end{align*}
For the second integral, we use \cref{middleprodbound} to get
\begin{align*}
\int_{-X}^0 |U_-(x) U_+(x)| \rmd x &\leq C \int_{-X}^0 \rme^{-\alpha(x + X)} \rme^{\alpha(x - X)} \rmd x \leq C \int_{-X}^0 \rme^{-\alpha(x + X)/2}\rme^{\alpha(x - X)} \rmd x \\
&\leq C \rme^{-(3 \alpha/2) X } \int_{-X}^0 \rme^{(\alpha/2)x} \rmd x \leq C \rme^{-(3 \alpha/2) X }
\end{align*}
The third and fourth integrals are similar. Combining these, we obtain \cref{expsepbound2}.
\end{proof}

\begin{remark}
If the hypotheses of \cref{expseplemma} are satisfied, we say that $U_+(x)$ and $U_-(x)$ are exponentially separated by $2X$.
\end{remark}

Next, we obtain a bound on the matrix $\vK_1$.
% lemma : K1 is small
\begin{lemma}\label{K1small}
For the matrix $\vK_1$ in \cref{Kreinform},
\begin{equation}\label{K1final}
\vK_1 = \mathcal{O}(\rme^{-(3 \alpha/2) X_{\mathrm{min}}}).
\end{equation}
\end{lemma}

\begin{proof}
Substituting $\calA_1 = -2c\partial_x$ into \cref{defK1}, $(\vK_1)_{jk} = \rmi 2 c \langle s_j, \partial_xs_k \rangle$. Using the expansion \cref{sj} from Theorem \ref{multiexist},
\begin{equation}\label{K1exp}
\begin{aligned}
\langle s_j ,\partial_xs_k \rangle
&= \sum_{m = 1}^{n} d_{jm} d_{km} \langle \partial_xU^m, \partial_x^2U^m \rangle
+ \sum_{m \neq\ell} d_{jm} d_{k\ell} \langle \partial_xU^m, \partial_x^2U^\ell \rangle\\
&\qquad
+ \langle s_j, \partial_x w_k \rangle
+ \sum_{\ell = 1}^{n} d_{k\ell} \langle w_j, \partial_x^2U^\ell \rangle.
\end{aligned}
\end{equation}
By translation invariance of the inner product on $L^2(\R)$,
\[
\langle \partial_xU^m, \partial_x^2U^m \rangle = \langle \partial_xU, \partial_x(\partial_xU) \rangle = 0,
\]
since the operator $\partial_x$ is skew-symmetric. For $m \neq\ell$, $U^m$ and $U^\ell$ are exponentially separated by at least $2 X_{\mathrm{min}}$; thus, by Lemma \ref{expseplemma},
\[
\langle \partial_xU^m, \partial_x^2U^\ell \rangle = \mathcal{O}(\rme^{-(3 \alpha/2) X_{\mathrm{min}}}).
\]
The last two terms in \cref{K1exp} are $\mathcal{O}(\rme^{-2 \alpha X_{\mathrm{min}}})$ using H\"{o}lder's inequality and the bound \cref{sjwbound} from \cref{multiexist}, which applies to $\partial_x w_k$ as well as $w_j$. Combining these estimates we obtain \cref{K1final}.
\end{proof}

Using the expansion \cref{sj} from \cref{multiexist}, the matrix $\vK_2$ in \cref{Kreinform} becomes,
\begin{equation}\label{K2expansion}
\begin{aligned}
&(\vK_2)_{jk}
= 4 c^2 \left\langle\sum_{m = 1}^{n} d_{jm} \partial_x^2U^m + \partial_x w_j,\right. \\
&\qquad\left.\sum_{\ell = 1}^{n} d_{k\ell} P_{S^\perp} (P_{S^\perp} \calA_0(U_n)P_{S^\perp})^{-1} P_{S^\perp} \partial_x^2U^\ell + P_{S^\perp} (P_{S^\perp} \calA_0(U_n)P_{S^\perp})^{-1} P_{S^\perp}\partial_xw_k \right\rangle.
\end{aligned}
\end{equation}
Before we can evaluate this expression, we need to look at $(P_{S^\perp} \calA_0(U_n)P_{S^\perp})^{-1}$.

% lemma : P_{S^\perp} \calA_0(U_n) |_{S^\perp} invertible

\begin{lemma}\label{PA0inv}
$P_{S^\perp} \calA_0(U_n) P_{S^\perp}: S^\perp \rightarrow S^\perp$ is an invertible linear operator with bounded inverse.
\end{lemma}

\begin{proof}
By \cref{A0ess}, the essential spectrum of $\calA_0(U_n)$ is $\sigma_{\text{ess}} = [1, \infty)$, which is bounded away from 0.
Thus the operator $\calA_0(U_n)$ is Fredholm with index 0. Since for the small magnitude eigenvalues $\nu_i$ of $\calA_0(U_n)$ we have $\nu_i \notin [1, \infty)$, the operator $A_0(U_n)  - \nu_i I$ is also Fredholm with index 0. Since $A_0(U_n)  - \nu_i I$ is Fredholm, its range is closed. Thus by the closed range theorem \cite[p.~205]{Yosida}, since $\nu_i \in \R$ and $A_0(q_n)$ is self-adjoint, we have
\begin{equation}\label{KerRangeNu}
\Ran (A_0(q_n) - \nu_i I) = \left(\ker (A_0(q_n) - \nu_i I)\right)^\perp.
\end{equation}

Next, we look at the operator $P_{S^\perp} \calA_0(U_n)$. Since $\calA_0(U_n)$ is self-adjoint and $P_{S^\perp}$ commutes with $\calA_0(U_n)$, $P_{S^\perp} \calA_0(U_n)$ is also self-adjoint. Since $P_{S^\perp} A_0(U_n) = A_0(U_n) P_{S^\perp}$, the kernel of $P_{S^\perp} \calA_0(U_n)$ contains $S$ as well as the kernel of $\calA_0(U_n)$, which is contained in $S$. The only other elements in the kernel of $P_{S^\perp} \calA_0(U_n)$ are functions $y$ for which $(A_0(q_n) - \nu_i I) y = s_i$, since that will be annihilated by the projection $P_{S^\perp}$. But such a function cannot exist, since by \eqref{KerRangeNu}, we would have $s_i \perp \ker (A_0(q_n) - \nu_i I)$, which contains $s_i$. We conclude that $\ker P_{S^\perp} A_0(q_n) = S$.

Since the range of $A_0(q_n)$ is closed and $P_{S^\perp}$ is bounded, the range of $P_{S^\perp} A_0(q_n)$ is also closed. Thus by the closed range theorem and the fact that $P_{S^\perp} A_0(q_n)$ is self-adjoint,
\[
\Ran P_{S^\perp} A_0(q_n) = (\ker (P_{S^\perp} A_0(q_n))^*)^\perp = (\ker (P_{S^\perp} A_0(q_n)))^\perp = S^\perp.
\]
Since $\dim \ker P_{S^\perp} A_0(q_n) = \text{codim } P_{S^\perp} A_0(q_n) = 2$, the operator $P_{S^\perp} A_0(q_n)$ is a Fredholm operator with index 0 and kernel $S$.

Thus the restriction $P_{S^\perp} \calA_0(U_n)|_{S^\perp} = P_{S^\perp} \calA_0(U_n) P_{S^\perp}$ is invertible on $S^\perp$. By the definition of $S$ and \cref{multiexist}, $P_{S^\perp}\calA_0(U_n)P_{S^\perp}$ has no eigenvalues of magnitude less than $\delta$. By the resolvent bound for normal operators, the linear operator $(P_{S^\perp} \calA_0(U_n)P_{S^\perp})^{-1}$ is bounded on $S^\perp$.
\end{proof}

Before we can evaluate the term $(P_{S^\perp} \calA_0(U_n)P_{S^\perp})^{-1} P_{S^\perp}\partial_x^2U^\ell$ from \cref{K2expansion}, we will need the following lemma which gives an expansion for $e^{U_n(x)}$.

% lemma : separation of exponential e^{U_n(x)}

\begin{lemma}\label{expsep}
For the $n-$pulse, $U_n(x)$, and for all $i = 1, \dots, n$,
\[%begin{equation}\label{expUnexpansion}
\exp(U_n(x)) = \exp( U^i(x)) + \sum_{j \neq i} (\exp(U^j(x)) - 1) + \mathcal{O}(\rme^{-\alpha X_{\mathrm{min}}})
\]%end{equation}
\end{lemma}

\begin{proof}
Fix $i$ in the expansion \eqref{qn} and let $S(x) = \sum_{j \neq i} U_j(x)$, so that $U_n = U^i + S + \mathcal{O}(e^{-\alpha X_{\mathrm{min}}})$. Since $U_n(x)$ is bounded,
\[
\begin{aligned}
\exp(U_n(x)) &= \exp( U^i(x) )\exp(S(x))(1 + \mathcal{O}(\rme^{-\alpha X_{\mathrm{min}}})) \\
&= \exp( U^i(x) )\exp(S(x)) + \mathcal{O}(\rme^{-\alpha X_{\mathrm{min}}}).
\end{aligned}
\]
Using the Taylor expansion for the exponential,
\[
\begin{aligned}
\exp( U^i(x) )\exp(S(x))
&= \sum_{m=0}^\infty \frac{U^i(x)^m}{m!}
\sum_{n=0}^\infty \frac{S(x)^n}{n!} \\
&= \sum_{m=0}^\infty \frac{U^i(x)^m}{m!}
+ \sum_{n=0}^\infty\frac{S(x)^n}{n!} - 1 +
\sum_{m=1}^\infty \frac{U^i(x)^m}{m!}
\sum_{n=1}^\infty \frac{S(x)^n}{n!} \\
&= \exp(U^i(x)) + \exp(S(x)) - 1 +
\sum_{m=1}^\infty \frac{U^i(x)^m}{m!}
\sum_{n=1}^\infty \frac{S(x)^n}{n!}
\end{aligned}
\]
For the last term on the RHS,
\[
\begin{aligned}
\left| \sum_{m=1}^\infty \frac{U^i(x)^m}{m!} \sum_{n=1}^\infty \frac{S(x)^n}{n!} \right|
&= \left| U^i(x)S(x)\right| \sum_{m=0}^\infty \frac{|U^i(x)|^m}{(m+1)!} \sum_{n=0}^\infty \frac{|S(x)|^n}{(n+1)!} \\
&\leq \left| U^i(x)S(x) \right| e^{|U^i(x)|}e^{|S(x)|} \\
&\leq C \rme^{-2 \alpha X_{\mathrm{min}}},
\end{aligned}
\]
where in the last line we used the fact that $U_n(x)$ is bounded together with the bound \cref{expsepbound1} from \cref{expseplemma}, since $U^i$ and each peak in $S$ are exponentially separated. Combining all of this,
\[
\begin{aligned}
\exp(U_n(x)) &= \exp(U^i(x)) + \exp(S(x)) - 1 + \mathcal{O}(\rme^{-\alpha X_{\mathrm{min}}})
\end{aligned}
\]
Repeat this procedure $n - 2$ more times to get the result.
\end{proof}

We can now evaluate $(P_{S^\perp} \calA_0(U_n) P_{S^\perp})^{-1} P_{S^\perp}\partial_x^2U^\ell$.

% lemma : evaluation of (P_{S^\perp} \calA_0(U_n)|_{S^\perp})^{-1} P_{S^\perp} q^i_{xx}

\begin{lemma}\label{PA0invqxx}
\begin{equation}\label{invqxx}
(P_{S^\perp} \calA_0(U_n)P_{S^\perp})^{-1} P_{S^\perp}\partial_x^2U^\ell = -\frac{1}{2c}P_{S^\perp}\partial_cU^\ell
+ \mathcal{O}(\rme^{-2 \alpha X_{\mathrm{min}}}).
\end{equation}
\end{lemma}

\begin{proof}
Let $y = (P_{S^\perp} \calA_0(U_n)P_{S^\perp})^{-1} P_{S^\perp}\partial_x^2U^\ell$. By Lemma \ref{PA0inv}, this is well-defined, and $y \in S^\perp$. Since $P_{S^\perp}\partial_x^2U^\ell$ is smooth and $(P_{S^\perp} \calA_0(U_n)P_{S^\perp})^{-1}$ is bounded, $y$ is smooth as well and is the unique solution to the equation
\begin{equation*}
(P_{S^\perp} \calA_0(U_n) P_{S^\perp})y = P_{S^\perp}\partial_x^2U^\ell,
\end{equation*}
which simplifies to
\begin{equation}\label{Linstart}
P_{S^\perp} \calA_0(U_n) y = P_{S^\perp}\partial_x^2U^\ell,
\end{equation}
since $y \in S^\perp$. Using Lin's method as in \cite{sandstede:som98}, we will look for a solution to \cref{Linstart} of the form,
\begin{equation}\label{Linsolform}
\tilde{y} = -\frac{1}{2c} P_{S^\perp}\partial_cU^\ell + \tilde{w},
\end{equation}
where $\tilde{w} \in S^\perp$. This ansatz is suggested by
\begin{equation}\label{uc}
\calA_0(U) \partial_c U = -2 c\partial_x^2 U,
\end{equation}
which we obtain by taking $u = U$ in equation \cref{eqODE} and differentiating with respect to $c$, which we can do since $U$ is smooth in $c$ by \cref{Uexistshyp}. Substituting \cref{Linsolform} into \cref{Linstart} and simplifying, we have
\begin{equation}\label{Lin2}
P_{S^\perp} \calA_0(U_n) \left(-\frac{1}{2c} \partial_cU^\ell \right) + P_{S^\perp} \calA_0(U_n) \tilde{w} = P_{S^\perp}\partial_x^2U^\ell.
\end{equation}
Using \cref{expsep}, for $j = 1, \dots, n$ we can write the operator $\calA_0(U_n)$ as,
\begin{equation}\label{A0expansion}
\calA_0(U_n) = \calA_0(U^\ell) + \sum_{k \neq \ell} (\rme^{U^k(x)} - 1) + \tilde{h}(x),
\end{equation}
where $\tilde{h}(x)$ is small remainder term with uniform bound $\|\tilde{h}\|_\infty = \mathcal{O}(\rme^{-\alpha X_{\mathrm{min}}})$. Substituting \cref{A0expansion} into the first term on the LHS of \cref{Lin2},
\begin{align}\label{Lin3}
P_{S^\perp} \left( \calA_0(U^\ell) + \sum_{k \neq \ell} (\rme^{U^k(x)} - 1) + \tilde{h}(x) \right) \left(-\frac{1}{2c}\partial_cU^\ell \right) + P_{S^\perp} \calA_0(U_n) \tilde{w} &= P_{S^\perp}\partial_x^2U^\ell
\end{align}
Since \cref{uc} holds for $U = U^\ell$,
\begin{equation}
P_{S^\perp} \calA_0(U^\ell) \left( -\frac{1}{2c} \partial_c U^\ell \right) = P_{S^\perp}\partial_x^2 U^\ell,
\end{equation}
where we divided by $-2c$ and applied the projection $P_{S^\perp}$ on the left. Using this, equation \cref{Lin3} simplifies to
\begin{align}\label{Lin4}
\calA_0(U_n) \tilde{w} +
P_{S^\perp} \left( \sum_{k \neq \ell} (\rme^{U^k(x)} - 1) + \tilde{h}(x) \right) \left(-\frac{1}{2c}\partial_cU^\ell \right) &= 0,
\end{align}
where we use the fact that $P_{S^\perp}$ commutes with $\calA_0(U_n)$, since it is a spectral projection for $\calA_0(U_n)$, and that $\tilde{w} \in S^\perp$. Since $\partial_c U^\ell$ and $U^k$ are exponentially separated for $k \neq \ell$, using \cref{expseplemma} and the same argument as in the proof of \cref{expsep},
\begin{align}\label{Linbound1}
P_{S^\perp} \sum_{k \neq \ell} (\rme^{U^k(x)} - 1) + \tilde{h}(x) \left(-\frac{1}{2c}\partial_cU^\ell \right) = \mathcal{O}\left(e^{-\alpha X_{\min}}\right).
\end{align}
Since $\partial_c U^\ell$ is bounded and  $\|\tilde{h}\|_\infty = \mathcal{O}(\rme^{-\alpha X_{\mathrm{min}}})$,
\begin{align}\label{Linbound2}
P_{S^\perp} \tilde{h}(x) \left(-\frac{1}{2c}\partial_cU^\ell \right) &=  \mathcal{O}\left(e^{-\alpha X_{\min}}\right).
\end{align}
Using \cref{Linbound1} and \cref{Linbound2}, equation \cref{Lin4} simplifies to the equation for $\tilde{w}$
\begin{equation}\label{A0heq}
\calA_0(U_n) \tilde{w} + h(x) = 0,
\end{equation}
where $h(x)$ is a small remainder term with uniform bound $\|h(x)\|_\infty = \mathcal{O}(\rme^{-\alpha X_{\mathrm{min}}})$.

We now follow the procedure in \cite{sandstede:som98}, which we briefly outline below. Let $W = (\tilde{w}, \partial_x\tilde{w},\partial_x^2 \tilde{w},\partial_x^3 \tilde{w})$. As in \cite{sandstede:som98}, we rewrite \cref{A0heq} as a first-order system for $W$, and we take $W$ to be a piecewise function consisting of the $2n$ pieces $W_j^\pm, j = 1, \dots, n$, where
\begin{align*}
W_j^-(x) &\in C^0([-X_{j-1}, 0]) \\
W_j^+(x) &\in C^0([0, X_j])
\end{align*}
with $X_0 = X_n = \infty$.
We note the the domains of the functions $W_j^\pm(x)$ overlap at the endpoints; the second and third equations in the system \cref{Wsystem} are matching conditions for these pieces at the appropriate endpoints.
Following this procedure, and using the expansions \cref{A0expansion} for $\calA_0(U_n)$ on the $j$-th piece, we obtain the system of equations
\begin{equation}\label{Wsystem}
\begin{aligned}
(W_j^\pm)'(x) = A(U(x)) W_j^\pm(x) &+ G_j(x) W_j^\pm(x)+ H_j(x) \\
W_j^+(X_i) - W_{j+1}^-(-X_j) &= 0  \\
W_j^-(0) - W_j^+(0) &= 0
\end{aligned}
\end{equation}
where
\[
A(U(x)) = \begin{pmatrix}
0 & 1 & 0 & 0 \\
0 & 0 & 1 & 0 \\
0 & 0 & 0 & 1 \\
-e^{U(x)} & 0 & -c^2 & 0
\end{pmatrix}, \quad
G_j(x) = \begin{pmatrix}
0 & 0 & 0 & 0 \\
0 & 0 & 0 & 0 \\
0 & 0 & 0 & 0 \\
\sum_{k \neq j} (1 - e^{U(x - \rho_{kj})}) & 0 & 0 & 0
\end{pmatrix},
\]
and $\rho_{kj}$ is the signed distance from peak of $U^k$ to peak of $U^j$ in $U_n$.
$H_j$ is a remainder term which comes from the term $h(x)$ in \cref{A0heq} and the remainder term in the expansion \cref{A0expansion}, and we have the estimate $\|H_j \|_\infty = \mathcal{O}(\rme^{-\alpha X_{\mathrm{min}}})$.
For $k \neq j$, $|\rho_{kj}| \geq 2 X_{\mathrm{min}}$. This implies $e^{U(x - \rho_{kj})} = \mathcal{O}(\rme^{-\alpha X_{\mathrm{min}}})$ on the $j$-th piece, thus we can use a Taylor expansion to show $\|G_j\| = \mathcal{O}(\rme^{-\alpha X_{\mathrm{min}}})$. Following the procedure in \cite{sandstede:som98}, we obtain a unique piecewise solution $W_j^\pm$ to the first two equations of \cref{Wsystem}. The third equation is generally not satisfied, so what we have constructed is a unique solution $\tilde{y}$ of the form \cref{Linsolform} to \cref{Linstart} which is continuous except for $n - 1$ jumps. By uniqueness, we must have $\tilde{y} = y$, thus $y$ is actually of the form \cref{Linsolform} with $\tilde{w}$ smooth. Finally, Lin's method gives us the uniform bound
$\|\tilde{w}\|_\infty = \mathcal{O}(\rme^{-2 \alpha X_{\mathrm{min}}})$, from which \cref{invqxx} follows.
\end{proof}

We prove one more lemma before we evaluate the matrix $\vK_2$ from \cref{Kreinform}.

% lemma : orthogonality of coefficients d_{jk}

\begin{lemma}\label{orthogonald}
For the coefficients $d_{jk}$ in \cref{sj} from \cref{multiexist},
\begin{equation}\label{dsum}
\sum_{m = 1}^{n} d_{jm} d_{km} = \delta_{jk} + \mathcal{O}(\rme^{-(3 \alpha/2) X_{\mathrm{min}}}).
\end{equation}
\end{lemma}

\begin{proof}
% I fixed stuff here, 12/27
Using the expansion \cref{sj} from \cref{multiexist},
\[
\begin{aligned}
\langle s_j, s_k \rangle
&= \sum_{m = 1}^{n} d_{jm} d_{km} \langle \partial_xU^m,\partial_xU^m \rangle
+ \sum_{m \neq \ell} d_{jm} d_{k\ell} \langle \partial_xU^m, \partial_xU^\ell \rangle\\
&\qquad
+ \langle s_j, w_k \rangle
+ \sum_{\ell = 1}^{n} d_{k\ell} \langle w_j, \partial_xU^\ell \rangle.
\end{aligned}
\]
As in \cref{K1small}, the second term on the RHS is $\mathcal{O}(\rme^{-(3 \alpha/2) X_{\mathrm{min}}})$, and the last two terms on the RHS are $\mathcal{O}(\rme^{-2 \alpha X_{\mathrm{min}}})$. By translation invariance, $\langle \partial_xU^m, \partial_xU^m \rangle = \langle \partial_xU, \partial_xU \rangle = \|\partial_xU\|^2$ for all $m$. This reduces to
\[
\langle s_j, s_k \rangle
= \|\partial_xU\|^2 \sum_{m = 1}^{n} d_{jm} d_{km} + \mathcal{O}(\rme^{-(3 \alpha/2) X_{\mathrm{min}}}).
\]
Dividing by $\|\partial_xU\|^2$ and using the orthogonality relation \cref{orthoeigs} gives us \cref{dsum}.
\end{proof}

Finally, we can evaluate the matrix $\vK_2$ from \cref{Kreinform}.

% lemma : K2 approx diagonal

\begin{lemma}\label{K2diag}
For the matrix $\vK_2$ in \cref{Kreinform},
\begin{equation}\label{K2final}
(\vK_2)_{jk}
= -2 c\langle\partial_x^2 U, \partial_c U \rangle \delta_{jk} + \mathcal{O}(\rme^{-(3 \alpha/2) X_{\mathrm{min}}}).
\end{equation}
\end{lemma}

\begin{proof}
By \cref{PA0inv}, $(P_{S^\perp} \calA_0(U_n)P_{S^\perp})^{-1}$ is a bounded linear operator. Using the bound \cref{sjwbound} from \cref{multiexist},
\[
P_{S^\perp} (P_{S^\perp} \calA_0(U_n)|_{S^\perp})^{-1} P_{S^\perp}\partial_xw_k = \mathcal{O}(\rme^{-2 \alpha X_{\mathrm{min}}}).
\]
Using this and \cref{invqxx} from \cref{PA0invqxx}, \cref{K2expansion} becomes,
\[
\begin{aligned}
(\vK_2)_{jk}
&= 4 c^2 \left\langle \sum_{m = 1}^{n} d_{jm}\partial_x^2U^m + \partial_xw_j,
-\frac{1}{2c}\sum_{\ell = 1}^{n} d_{k\ell} P_{S^\perp}\partial_cU^\ell + \mathcal{O}(\rme^{-2 \alpha X_{\mathrm{min}}}) \right\rangle \\
&= -2 c \left( \sum_{m = 1}^{n} d_{jm} d_{km} \langle \partial_x^2U^m, P_{S^\perp} \partial_cU^m \rangle
+ \sum_{m\neq \ell} d_{jm} d_{k\ell} \langle \partial_x^2U^m, P_{S^\perp} \partial_cU^l \rangle\right.\\
&\qquad\qquad\left.+ \sum_{\ell=1}^n \langle \partial_xw_j, d_{k\ell}\partial_cU^\ell \rangle \right)
 + \mathcal{O}(\rme^{-2 \alpha X_{\mathrm{min}}}).
\end{aligned}
\]
By \cref{Yexploc} and \cref{lemma:Ycexploc}, $\partial_x^2 U$ and $\partial_c U$ are exponentially localized, thus for $m\neq\ell$, $\partial_x^2U^m$ and $\partial_cU^\ell$ are exponentially separated. It follows from \cref{expseplemma} that the second term on the RHS is $\mathcal{O}(\rme^{-(3 \alpha/2) X_{\mathrm{min}}})$.
Using H\"{o}lder's inequality and the remainder bound \cref{sjwbound}, the third term on the RHS is $\mathcal{O}(\rme^{-2 \alpha X_{\mathrm{min}}})$. Thus we are left with
\begin{equation}\label{K2step1}
\begin{aligned}
(\vK_2)_{jk}
&= -2 c \sum_{m = 1}^{n} d_{jm} d_{km} \langle \partial_x^2U^m, P_{S^\perp} \partial_cU^m \rangle + \mathcal{O}(\rme^{-(3 \alpha/2) X_{\mathrm{min}}}).
\end{aligned}
\end{equation}
To evaluate the inner product, we first evaluate $P_S \partial_c U^m$. Recalling the normalization \cref{orthoeigs} and using the expansion \cref{sj}, since the $s_j$ are orthogonal,
\[
\begin{aligned}
P_S &\partial_c U^m = \frac{1}{\|\partial_x U\|} \sum_{j=1}^n \langle s_j, \partial_c U^m \rangle \\
&= \frac{1}{\|\partial_x U\|} \sum_{j=1}^n \sum_{k=1}^n \langle d_{jk} \partial_x U^k + w_k, \partial_c U^m \rangle \\
&= \frac{1}{\|\partial_x U\|} \left( \sum_{j=1}^n d_{jm} \langle \partial_x U^m, \partial_c U^m \rangle + \sum_{j=1}^n \sum_{k \neq m}^n d_{jk} \langle \partial_x U^k, \partial_c U^m \rangle \right) + \mathcal{O}(\rme^{-2 X_{\mathrm{min}}}) \\
&= \frac{1}{\|\partial_x U\|} \sum_{j=1}^n d_{jm} \langle \partial_x U, \partial_c U \rangle + \mathcal{O}(\rme^{-(3 \alpha/2) X_{\mathrm{min}}}) \\
&= \mathcal{O}(\rme^{-(3 \alpha/2) X_{\mathrm{min}}}).
\end{aligned}
\]
The third line follows from \cref{expseplemma}, since by \cref{Yexploc} and \cref{lemma:Ycexploc}, $\partial_x U$ and $\partial_c U$ are exponentially localized, thus $\partial_x U^k$ and $\partial_c U^m$ are exponentially separated for $k \neq m$.
In the fourth line we use $\langle \partial_x U, \partial_c U \rangle = 0$, since $\partial_x U$ is an odd function and $\partial_c U$ is an even function. From this, we have
\[
P_{S^\perp} \partial_c U^m = (\calI - P_S) \partial_c U^m = \partial_c U^m + \mathcal{O}(\rme^{-(3 \alpha/2) X_{\mathrm{min}}}).
\]
Substituting this into equation \cref{K2step1} and using \cref{orthogonald} and translation invariance, this becomes
\[
\begin{aligned}
(\vK_2)_{jk}
&= -2 c \sum_{m = 1}^{n} d_{jm} d_{km} \langle \partial_x^2U^m, \partial_cU^m \rangle
= -2 c \langle \partial_x^2U, \partial_cU \rangle \sum_{m = 1}^{n} d_{jm} d_{km} \\
&= -2 c \langle \partial_x^2U,\partial_cU \rangle \delta_{jk} + \mathcal{O}(\rme^{-(3 \alpha/2) X_{\mathrm{min}}}),
\end{aligned}
\]
which is \cref{K2final}.
\end{proof}

Using \cref{K1final} from \cref{K1small} and \cref{K2final} from \cref{K2diag}, the Krein matrix \cref{Kreinform} becomes,
\[%\begin{equation}\label{Kreinform2}
\begin{aligned}
-\frac{\vK_S(z)}{z}&= \|\partial_xU\|^2 \text{diag}(\nu_1, \dots, \nu_n)
- ( \|\partial_xU\|^2 -2 c \langle \partial_x^2U, \partial_cU \rangle) \vI_n \overline{z}^2 \\
&\qquad + \mathcal{O}(\rme^{-(3 \alpha/2) X_{\mathrm{min}}}|z| + |z|^3).
\end{aligned}
\]%\end{equation}
Integrating by parts,
\[
\begin{aligned}
-\frac{\vK_S(z)}{z}
&= \|\partial_xU\|^2 \text{diag}(\nu_1, \dots, \nu_n) - \left( \langle \partial_xU, \partial_xU \rangle + 2c\langle \partial_c\partial_xU, \partial_xU\rangle \right)\vI_n\overline{z}^2\\
&\qquad + \mathcal{O}(\rme^{-(3 \alpha/2) X_{\mathrm{min}}}|z| + |z|^3)  \\
&= \|\partial_xU\|^2 \text{diag}(\nu_1, \dots, \nu_n) -\partial_c\left( c||\partial_xU||^2 \right) \vI_n \overline{z}^2  + \mathcal{O}(\rme^{-(3 \alpha/2) X_{\mathrm{min}}}|z| + |z|^3) \\
&= \|\partial_xU\|^2 \text{diag}(\nu_1, \dots, \nu_n) + d''(c) \vI_n \overline{z}^2  + \mathcal{O}(\rme^{-(3 \alpha/2) X_{\mathrm{min}}}|z| + |z|^3),
\end{aligned}
\]
which is \cref{Kreinapprox} in \cref{Kreindiag}.

%%\phantomsection                                         % necessary to add appendix section to bookmarks using hyperref
%%\addcontentsline{toc}{section}%                         % add appendix title to table of contents
%%  {Appendix B: Spectral analysis for a square well potential}
%%\setcounter{section}{2} %
%%\setcounter{theorem}{0} %
%%\setcounter{equation}{0} %
%%\input{AppendixB3W4}
%%\phantomsection                                         % necessary to add appendix section to bookmarks using hyperref
%%\addcontentsline{toc}{section}%                         % add appendix title to table of contents
%%  {Appendix C: Exact roots of a polynomial of degree 6}
%%\setcounter{section}{3} %
%%\setcounter{theorem}{0} %
%%\setcounter{equation}{0} %
%%\input{AppendixC3W4}
%
%%%%%%%%%%%%%%%%%%%%%%%%%%%%%%%%%%%%%%%%%%%%%%%%%%%%%%%%%%%%%%%%%%%%%%%%%
%
%%\addcontentsline{toc}{chapter}{\bibname}                % for books
%%\addcontentsline{toc}{section}{\refname}                % for articles
%
%%\begin{center}                                          % new heading for references
%%\textsc{References}
%%\end{center}
%%{\renewcommand\section[2]{}%                            % references without the heading
%%   \bibliography{../../papers}}                         %

% standard reference style
%\phantomsection                                         % necessary to add reference section to bookmarks using hyperref
%\addcontentsline{toc}{section}{\refname}                % add reference title to table of contents

%\bibliography{../../../Dropbox/papers}
%\bibliography{../../../../papers}
\bibliographystyle{siamplain}

%%%%%%%%%%%%%%%%%%%%%%%%%%%%%%%%%%%%%%%%%%%%%

\end{document}